\documentclass[12pt]{amsart}
\usepackage{amsfonts}
\usepackage{amssymb}
\usepackage[utf8x]{inputenc}
\usepackage{xspace}
\usepackage[inline,shortlabels]{enumitem}
\usepackage[a4paper,margin=3cm]{geometry}
\usepackage{mathtools}

\usepackage{booktabs}
 \usepackage[bb=boondox]{mathalfa}
 \newcommand{\maxzero}{\boldsymbol{\varepsilon}}
 \newcommand{\maxone}{\mathbb{1}}

\newtheorem{theorem}{Theorem}[section]
\newtheorem{corollary}[theorem]{Corollary}

\newtheorem{proposition}[theorem]{Proposition}

\theoremstyle{definition}
\newtheorem{example}[theorem]{Example}

\newtheorem{definition}[theorem]{Definition}

\newtheorem{remark}[theorem]{Remark}
\newtheorem{remarks}[theorem]{Remarks}

\def\Cc{\hbox{\sf C\kern -.47em {\raise .48ex \hbox{$\scriptscriptstyle |$}}
   \kern-.5em {\raise .48ex \hbox{$\scriptscriptstyle |$}} }}

\newcommand{\be}{\begin{equation}}
\newcommand{\ee}{\end{equation}}

\newcommand{\cP}{{\mathcal P}}

\newcommand{\RR}{\mathbb{R}}

\newcommand{\half}{\frac{1}{2}}
\newcommand{\Rmax}{\RR_{\max}}
\newcommand{\Rmin}{\RR_{\min}}
\newcommand{\Rxmax}{\RR_{\max}[x]}
\newcommand{\Rnnmax}{\Rmax^{n \times n}}
\DeclareMathOperator{\perm}{\mathrm{perm}}

\newcommand{\SG}{\mathfrak{S}}

\newcommand{\hhalf}[1]{#1^{\circ \half}}
\newcommand{\fullchi}[1]{\bar{\chi}_{#1}}

\newcommand{\norm}[1]{\lVert #1 \rVert}
\newcommand{\abs}[1]{\lvert #1 \rvert}

\newcommand{\markentry}[1]{{#1}^\star{}}

\numberwithin{equation}{section}

\newcommand{\maxconv}[1]{\boxplus_{#1}}

\usepackage{color}

\begin{document}

\title[]{Polynomial convolutions in max-plus algebra}

\author{Amnon Rosenmann}
\address[A. Rosenmann and F. Lehner]{Institute of Discrete Mathematics,
	Graz University of Technology, Steyrergasse 30, A-8010 Graz, Austria}
\email[A.~Rosenmann]{rosenmann@math.tugraz.at}
\author{Franz Lehner}
\email[F.~Lehner]{lehner@math.tu-graz.ac.at}
\author{Aljo\v{s}a Peperko}
\address[A. Peperko]{Faculty of Mechanical Engineering, University of Ljubljana, A\v{s}ker\v{c}eva 6, SI-1000 Ljubljana, Slovenia; 
 Institute of Mathematics, Physics and Mechanics, Jadranska 19, SI-1000 Ljubljana, Slovenia. }
\email[A.~Peperko]{aljosa.peperko@fmf.uni-lj.si}

\date{\today}
\maketitle

\begin{abstract}
Recently, in a work that grew out of their exploration of interlacing polynomials, Marcus, Spielman and Srivastava \cite{MSS15} and Marcus \cite{Marcus16} studied certain combinatorial polynomial convolutions.
These convolutions preserve real-rootedness and capture expectations of characteristic polynomials of unitarily invariant random matrices, thus providing a link to free probability.
We explore analogues of these types of convolutions in the setting of max-plus algebra.
In this setting the max-permanent replaces the determinant, the maximum is the analogue of the expected value and real-rootedness is replaced by full canonical form.
Our results resemble those of Marcus et al., however, in contrast to the classical setting we obtain an exact and simple description of all roots of the convolution of $p(x)$ and $q(x)$ in terms of the roots of $p(x)$ and $q(x)$.
\end{abstract}

\noindent
\\
\section{Introduction}
The study of polynomials, their roots and their critical points from the algebraic, analytic and geometric point of view has a long history (see, e.g., the monographs \cite{Dieudonne38,Marden66,RS02}, which emphasize the analytic-geometric approach).
Recently,  Marcus, Spielman and Srivastava \cite{MSS15} and Marcus
\cite{Marcus16} initiated the study of certain convolutions of polynomials that can implicitly be found in a paper by Walsh \cite{Walsh22} from the early last century.
They established a strong link to free probability by showing that these convolutions capture the expected characteristic polynomials of random matrices. 
However, their initial motivation came from the study of interlacing polynomials, which led to the solution of the Kadison-Singer problem \cite{MSS15a}, and they showed that  these convolutions preserve the property of the roots being real numbers.

In the present paper we explore analogues of these types of convolution polynomials in the setting of max-plus algebra.
In this setting the max-permanent replaces the determinant and the maximum is the analogue of the expected value.
Our results resemble those of \cite{MSS15} in terms of the formulas for the convolution polynomials, but the formulas as well as the computations are simpler in the max-plus setting.
In addition, whereas in the classical setting only bounds on the maximal roots of the convolution polynomials are known, in max-plus algebra we obtain an exact and simple description of all the roots of the convolutions of maxpolynomials in terms of the roots of the involved maxpolynomials.
The preservation of the real-rootedness in the classical convolutions is represented by the preservation of the full canonical form in convolutions in max-plus algebra.

The paper is organized in the following way.
In Section~\ref{sec:maxplusalgebra}
we discuss maxpolynomials in full canonical form \cite{C-GM80}, that is,
maxpolynomials which are formally fully reducible to linear factors.
We continue in Section~\ref{sec:charpoly} with a description of the different types of characteristic maxpolynomials that we deal with:
In addition to the standard characteristic maxpolynomial we also consider the full characteristic maxpolynomial and the Gram characteristic maxpolynomial.
It turns out that the full characteristic maxpolynomial is always in full canonical form and therefore more appropriate for the questions considered in the present paper.

Next, we explore in Section~\ref{sec:maxconv} the convolution of characteristic maxpolynomials, which is the equivalent of the additive convolution of \cite{MSS15}.
The additive convolution is defined over random orthogonal matrices, and this definition allows, when working with symmetric matrices, to reduce the computations to diagonal matrices.
In the max-plus setting the set of orthogonal matrices consists only of permutation matrices, hence the computation of the maximum (the analogue of the expectation computation in the additive convolution) is over a finite set.
In fact, also in the classical setting it suffices to perform the computations over the set of signed permutation matrices (but, on the other hand, also over the set of unitary matrices).
Another feature of max-plus algebra is that we cannot achieve diagonalization due to the fact that there are no ``negative'' elements with respect to the max operation.
It turns out, however, that the set of ``principally dominant'' matrices, the matrices whose characteristic maxpolynomial equals the full characteristic maxpolynomial, suffices for our purpose.
When the matrices are not principally dominant then
the computation is executed with respect to the full characteristic maxpolynomial. 
Whereas in the standard additive convolution only a bound on the maximal root of the convolution polynomial can be given \cite{Walsh22}, the roots of the max convolution polynomial of degree $n$ are exactly the maximal roots among those of the involved
characteristic maxpolynomials.
We close Section~\ref{sec:maxconv}  with the ``max-row convolution'', which is the analogue of the ``asymmetric additive convolution'' of \cite{MSS15}.

The final section is about Hadamard product of characteristic maxpolynomials.
The main result relates the Hadamard product of Gram characteristic
maxpolynomials to full characteristic maxpolynomials of a product of permuted
matrices
and thus provides an analogue of the multiplicative convolution in standard arithmetic. 
Here the ordered list of the roots of the resulting maxpolynomial consists of the product (sum in standard arithmetic) of the ordered lists of the roots of the involved maxpolynomials, whereas in the standard multiplicative convolution only a bound on the maximal root of the convolution polynomial is given \cite{Szego22}.
\section{Max-plus algebra}
\label{sec:maxplusalgebra}
In its current setting, max-plus algebra is a relatively new field, which
emerged from several branches of mathematics simultaneously. 
It is an algebra over the ordered, idempotent semiring (in fact, semifield)
$\Rmax = \RR \cup \{ -\infty \}$, equipped with the operations of addition $a
\oplus b = \max(a,b)$ and multiplication $a \odot b = a+b$, with the unit
elements $\maxzero =-\infty$ (for addition) and $\maxone = 0$ (for multiplication).
As in standard arithmetic, the operations of addition and multiplication are associative and commutative, and multiplication is distributive over addition.
Matrix and polynomial operations are defined similarly to their standard counterparts, with the max-plus operations replacing the standard operations.

Max-plus algebra is isomorphic to min-plus algebra (also known as tropical algebra), which is the semifield $\Rmin = \RR \cup \{\infty \}$, where addition is replaced by minimum and multiplication  by addition, and also to max-times algebra $\RR_{+}$, where addition is replaced by maximum and multiplication is the same as in standard arithmetic. 
For more on max-plus algebra we refer to the monograph of Butkovi{\v{c}} \cite{Butkovic10}. 
Max-plus algebra is a part of a broader branch of mathematics, ``idempotent mathematics'', which was developed mainly by Maslov and his collaborators (see \cite{Litvinov07} for a brief introduction). 

Max-plus algebra, together with  its isomorphic versions, provides an attractive way of describing a class of non-linear problems appearing for instance in manufacturing and transportation scheduling, information technology, discrete event-dynamic systems, combinatorial optimization, mathematical physics and DNA analysis (see, e.g., \cite{Butkovic10,BCOQ92,PS05,LM05,Litvinov07,B98,MP15,BGC-G09}, and the references cited there). 

For the sake of readability, we mostly suppress the multiplication sign $\odot$, writing $ab$ instead of $a \odot b$ and  $ax^3$ instead of $a \odot x^{\odot 3}$ or $a \odot x \odot x \odot x$.
Also, when an indeterminate $x$ appears without a coefficient, as in $x^n$, then its coefficient is naturally the multiplicative identity element, i.e. $0$.

\bigskip

\subsection{Maxpolynomials and tropical roots}
A \textbf{(formal) maxpolynomial} is an expression of the form
\begin{equation}
  \label{eq:defmaxpolynomial}
p (x)= \bigoplus _{k=0} ^n a_k x^k = \max \{a_k + k x: k=0, 1, \ldots, n \},
\end{equation}
where $a_0 , \ldots, a_n \in  \Rmax$ and $x$ is a formal indeterminate.
We assume that $a_n \neq \maxzero$ and then $p(x)$ is of degree $n$, unless $p(x) = \maxzero x^0$, the null maxpolynomial, which is of degree $-\infty$.
The terms $a_kx^k$ are called the monomials constituting $p(x)$ and normally the monomial 
$a_kx^k$ is omitted when $a_k = \maxzero$.
The set of maxpolynomials form a semiring (there is no additive inverse) $\Rxmax$ with respect to the max-plus operations of addition and multiplication. 
Each expression in $\Rxmax$ that is the result of application of these max-plus operations to maxpolynomials can be reduced to a unique \textbf{canonical form} as in \eqref{eq:defmaxpolynomial} according to the rules of $\Rmax$.
Hence, we say that two maxpolynomials are (formally) equal if they have the same canonical form.  

A maxpolynomial $p(x)$ induces a convex, piecewise-affine function $\hat{p}(x)$ on $\Rmax$.
Unlike the situation in standard arithmetic, two distinct formal maxpolynomials $p_1(x)$ and $p_2(x)$ may represent the same polynomial function, that is, $\hat{p}_1(x) = \hat{p}_2(x)$ as functions.
The (\textbf{max-plus} or \textbf{tropical}) \textbf{roots} of a maxpolynomial $p(x)$ are the points at which $\hat{p}(x)$ is non-differentiable.
The multiplicity of a root equals the change of the slope of $\hat{p}(x)$ at that root.
Equivalently, the roots of $p(x)$ are the values $r \neq \maxzero$ of $x$ 
for which the maximum \eqref{eq:defmaxpolynomial} is attained at least twice, that is, there exist $i,j \in \{0, \ldots,n \}$ such that $i \neq j$ and for all $k \in \{0, \ldots,n \}$, $\hat{p}(r) = a_i r^i = a_j r^j \geq a_k r^k$.
The multiplicity of the root $r$ is $\max_{i,j} \{ | i-j | \, : \, a_i r^i = a_j r^j\}$.
We also count $\maxzero=-\infty$ as a root with multiplicity $l$ whenever $a_0, a_1, \ldots , a_{l-1}$ are all equal to $\maxzero$ and $a_l \neq \maxzero$.
In this case the monomials $\maxzero x^i$, $0 \leq i \leq l-1$, represent the constant function $\maxzero$, which intersects the line $\hat{p}_l(x) = a_l x^l$ at $x = \maxzero$.

\goodbreak{}
\subsection{Derivatives of maxpolynomials}
\begin{definition}
	Given a maxpolynomial $p(x) \in \Rxmax$, we define its (formal) \textbf{max-plus derivative} to be the result of applying the max-linear shift operator (or ``annihilation operator'')
$\partial_x : \Rxmax \to \Rxmax$ 
defined by 
\begin{equation*}
  \partial_x x^k =
  \begin{cases}
    x^{k-1} & k\geq 1\\
    \maxzero & k=0
  \end{cases}
\end{equation*}
i.e.,
$$
\partial_x\left( \bigoplus_{i=0}^{n} a_i x^i \right) = \bigoplus_{i=0}^{n-1} a_{i+1}  x^i
.
$$
\end{definition}
	We will use the notation $p'(x)$ for $\partial_xp(x)$ and $p^{(i)}(x)$ for $\partial_x^ip(x)$ wherever convenient.
This is a derivation (Proposition~\ref{prop:leibniz} below);
in fact, it is the unique derivation on $\Rmax[x]$ satisfying
\begin{enumerate}[(i)]
	\item $a' = \maxzero$ for every $a \in \Rmax$,
	\item $x' = 0$,
\end{enumerate}
i.e., the unique operator on $\Rmax[x]$ satisfying
\begin{enumerate}[(i),resume]
	\item $(p \oplus q)'(x) = p'(x) \oplus q'(x)$ (linearity),
	\item $(pq)'(x) = (p'q)(x) \oplus (p q')(x)$ (Leibniz's rule),
\end{enumerate}
for every $p(x), q(x) \in \Rxmax$.

One then shows (e.g., by induction) that the following iterated Leibniz's rule holds:
\begin{proposition}[Leibniz's rule]
  \label{prop:leibniz}
	Let $p(x), q(x) \in \Rxmax$. Then
	$$
	(pq)^{(k)}(x) = \bigoplus_{i=0}^{k} p^{(i)}(x) \, q^{(k-i)}(x).
	$$
\end{proposition}
\begin{remarks}
  \begin{enumerate}[1.]
   \item 
	Proposition~\ref{prop:rootsderivative} below asserts that roughly speaking, the derivative of a maxpolynomial function can be defined as the operation of removing the smallest root (and, in general, a possible reduction in the value of other smallest roots).
   \item 
    We emphasize that our derivation operates on a purely formal level.
    The derivatives of different maxpolynomial $p(x)$ representatives of a piecewise-linear convex function $\hat{p}(x)$ need not coincide, e.g., the maxpolynomials 
    $p(x) =  x^2 \oplus x \oplus 0$ and
    $q(x) = x^2 \oplus 0$
    are functionally equivalent: $\hat{p}(x) = \hat{q}(x)$, however, the derivatives $p'(x) = x \oplus 0$ and $q'(x)=x$ are not.
	In order for a derivative to respect functional equivalence we could choose to apply the above formal derivative not to $p(x)$ but to its FCF representative (see subsection~\ref{subsec:FCF} below).
	However, the maxpolynomials we are treating here are in FCF and the two definitions of a derivation coincide.
    \item 
	 For another form of derivative in tropical mathematics we refer to the ``layered derivative'', which is defined in \cite{IKR14}.
    \item 
    Given a maxpolynomial $p(x) = \bigoplus_{i=0}^n a_i x^i$, it can be represented via its Taylor expansion around $\maxzero$ as
    $$
    \bigoplus_{i=0}^n \widehat{p^{(i)}}(\maxzero) x^i,
    $$
    however, the Taylor expansion does not hold in other points.
  \end{enumerate}
\end{remarks}
\subsection{Full canonical maxpolynomials}
\label{subsec:FCF}
\begin{definition}
	A maxpolynomial $p(x) = \bigoplus_{i=0}^{n} a_i x^i$ is in \textbf{full canonical form} \cite{C-GM80} (or is an \textbf{FCF maxpolynomial}) if it is either constant or equal to the formal expansion of a product of linear factors: $p(x) = a_n (x \oplus s_1) \cdots (x \oplus s_n)$.
\end{definition}
The coefficients of FCF maxpolynomials form a concave sequence (these polynomials are also called \emph{concavified polynomials} or \emph{polynomials of maximum canonical form} \cite{BCOQ92}, \emph{maximally represented maxpolynomials} \cite{Tsai:2012:working} or \emph{least coefficient minpolynomials} in the min-plus setting \cite{GriggManwaring:2007:elementary}).
As before, we adopt the convention that monomials with coefficient $\maxzero$
are not written out and then it is clear that every monomial is in full canonical form; indeed $ax^n=a(x\oplus\maxzero)^n$.

If we partition the set of maxpolynomials in $\Rxmax$ according to their functional property then each equivalence class has a unique FCF maxpolynomial that represents it.
For example, $p(x) = x^4 \oplus x^3 \oplus x^2 = (x \oplus 0)^2 (x \oplus \maxzero)^2$ is the full canonical representative of the equivalence class of all maxpolynomials $q(x)$ with $\hat{q}(x) = \hat{p}(x)$ which contains, among others, the  maxpolynomials $x^4 \oplus x^2$ and $x^4 \oplus (-1) x^3 \oplus x^2$ which are not FCF.

In the following proposition we adopt the convention $-\infty - (-\infty) = -\infty$.
\begin{proposition}
	Let $p(x) = \bigoplus_{i=0}^{n} a_i x^i$ and let $r_1 \leq r_2 \leq \cdots \leq r_n$ be its roots.
	Then the following are equivalent.
	\begin{enumerate}[(i)]
		\item $p(x)$ is in full canonical form: $p(x) = a_n (x \oplus s_1) \cdots (x \oplus s_n)$ for some $s_1, \ldots, s_n$.
		\item\label{it:semiess:factorr}
                 $p(x) = a_n (x \oplus r_1) \cdots (x \oplus r_n)$.
		\item $r_i = a_{i-1} - a_i$, $i = 1, \ldots, n$.
		\item Concavity: $a_{i-1} - a_i \leq a_i - a_{i+1}$, $i = 1, \ldots, n-1$.
		\item $\hat{p}(r_i) = a_i r_i^i$ ($a_i + i r_i$ in standard arithmetic) for $i = 1, \ldots, n$.
	\end{enumerate}
\label{prop:fullcanonical}
\end{proposition}
\begin{proof}
	The equivalences follow from the definition of the tropical roots of a maxpolynomial and from simple manipulations of the above expressions.
	
	We demonstrate some of the implications.
	Suppose that (i) holds.  Without loss of generality we may assume
        that $s_1 \leq \cdots \leq s_n$.
	Expanding the expression $a_n (x \oplus s_1) \cdots (x \oplus s_n)$ 
        shows that the coefficients $a_k$ reduce to
        \begin{align*}
          a_k&= \bigoplus_{i_1<i_2<\dots<i_{n-k}} a_n s_{i_1}\ s_{i_2}\dots s_{i_{n-k}}\\
		  &= a_n s_{k+1} \dots s_n.
        \end{align*}
        It follows that $a_{k-1}-a_k=s_k$ and moreover
        \begin{align*}
		  \hat{p}(s_k) &= a_n (s_k \oplus s_1) \cdots (s_k \oplus s_k)(s_k \oplus s_{k+1}) \cdots (s_k \oplus s_n)  \\       
		  &= a_n  s_{k+1} \dots s_n s_{k}^k = a_k s_{k}^k \\
		  &= a_n s_{k} \dots s_n s_{k}^{k-1} = a_{k-1} s_k^{k-1}.\\
        \end{align*}
		Hence $\hat{p}(s_k) = a_k s_k^{k} = a_{k-1} s_k^{k-1}$ and 
		$x=s_k$ is a root for each $k$, i.e.,
        $s_k=r_k$ for each $k$ and thus (ii) and (iii) hold.

	As for $(v)$, geometrically, at each root $x = r_i$ of $p(x)$, the convex piecewise-linear function $\hat{p}(x)$ has a corner at the point $(r_i, \hat{p}(r_i) )$ ($\hat{p}(x)$ is non-differentiable at $x = r_i$), and when $p(x)$ is in full canonical form then when $r_i = r_{i+1} = \cdots = r_{i+l}$ all lines $a_{i+j} x^{i+j}$, $j =0,\ldots,l$, pass through this corner.
\end{proof} 
The sum of maxpolynomials in full canonical form need not be in full canonical form; however the next proposition shows that FCF polynomials are closed under
the operation of derivative and multiplication.
\begin{proposition}
	Let $p(x), q(x)$ be FCF maxpolynomials. Then
	\begin{enumerate}[(i)]
		\item $p'(x)$ is in FCF.
		\item $(pq)(x)$ is in FCF. 
	\end{enumerate}
\label{prop:qe}
\end{proposition}
\begin{proof}
	Let us assume that $p(x) = \bigoplus_{i=0}^{n} a_i x^i$ and $q(x) = \bigoplus_{i=0}^{m} b_i x^i$ are not constant, otherwise the proof follows immediately.
	Since $p(x)$ and $q(x)$ are FCF maxpolynomials they can be linearly factored: $p(x) = a_n (x \oplus r_1) \cdots (x \oplus r_n)$, $q(x) = b_m (x \oplus s_1) \cdots (x \oplus s_m)$.
	The product of the maxpolynomials is then $(pq)(x) = a_n b_m (x \oplus r_1) \cdots (x \oplus r_n) (x \oplus s_1) \cdots (x \oplus s_m)$, which is also a product of linear factors and hence FCF.
	
	The derivative of $p(x)$ is $p'(x) =\bigoplus_{i=1}^{n} a_i x^{i-1}$ and the differences $a_{i-1} - a_i$ (that satisfy the concavity condition) remain in $f'(x)$, except for $r_1 = a_0 - a_1$.
	It follows that $p'(x) = a_n (x \oplus r_2) \cdots (x \oplus r_n)$, which is again in full canonical form.
\end{proof} 
The next proposition shows that the derivative acts on FCF maxpolynomials
as a shift on the roots as well.
\begin{proposition}
  \label{prop:rootsderivative}
  Let $p(x)$ be an FCF maxpolynomial
  with roots $r_1\leq r_2\leq\dots\leq r_n$.
  Then the roots of $p'(x)$ are $r_2\leq r_3\leq\dots  \leq r_n$.
\end{proposition}
\begin{proof} 
The result follows from the proof of Proposition \ref{prop:qe}.
Alternatively, by applying Leibniz's rule to the factorization
  \begin{equation*}
    p(x) = a_n(x\oplus r_1)(x\oplus r_2)\dotsm(x\oplus r_n)
  \end{equation*}
 we  obtain
  \begin{equation*}
    p'(x) = a_n \bigoplus_{k=1}^n (x\oplus r_1)\dotsm(x\oplus r_{k-1}) (x\oplus r_{k+1}) \dotsm(x\oplus r_n),
  \end{equation*}
	that is, the term $(x\oplus r_k)$ is eliminated in the $k$-th summand.
  Now it is clear that the maximal value is attained when the smallest root is eliminated.
\end{proof}

\begin{corollary}
Let $p(x)$ be a maxpolynomial of degree $n$.
Then $p(x)$ is in FCF if and only if the roots of $p^{(k)}(x)$ are the
largest $n-k$ roots of $p(x)$, for $k=0,\ldots,n$.
\end{corollary}
\begin{proof} 
Let $p(x) = \bigoplus_{i=0}^{n} a_i x^i$ with roots $r_1\leq r_2\leq\dots\leq r_n$.
Suppose that $p(x)$ is in full canonical form.
Then by Proposition~\ref{prop:qe}, $p'(x)$ is also in FCF and
by Proposition~\ref{prop:rootsderivative} its roots are $r_2,\ldots,r_n$.
By induction, for $k=0,\ldots,n$, $p^{(k)}(x)$ is in FCF with roots $r_{k+1},\ldots,r_n$, the $n-k$ largest roots of $p(x)$.

Suppose now that $p(x)$ is not in FCF.
Then the concavity property of Proposition~\ref{prop:fullcanonical}~(iv) is not fulfilled.
That is, there exists $1 \leq k \leq n-1$ such that $a_{i-1} - a_i \leq a_i - a_{i+1}$ for $i = 1, \ldots, k-1$ but $a_{k-1} - a_k > a_k - a_{k+1}$.
It follows that $r_{k+1} > a_k - a_{k+1}$ (as property (iii) of Proposition~\ref{prop:fullcanonical} is not satisfied).
Then the roots of $p^{(k)}(x) = \bigoplus_{i=k}^{n} a_i x^i$ are $a_k - a_{k+1}, r_{k+2},\ldots,r_n$, which are not the largest $n-k$ roots of $p(x)$.
\end{proof}	
\begin{example}
	Let $p(x) = (-1)x^2 \oplus x \oplus 1 = (-1)(x \oplus 1)^2$ and $q(x) = x^2 \oplus x \oplus 0 = (x \oplus 0)^2$.
	Then $p'(x) = (-1)x \oplus 0 = (-1)(x \oplus 1)$ and $pq(x) = (-1) x^4 \oplus x^3 \oplus 1x^2 \oplus 1x \oplus 1 = (-1)(x \oplus 0)^2 (x \oplus 1)^2$. Hence, $p(x), q(x), p'(x)$ and $(pq)(x)$ can be decomposed into linear factors and thus are FCF maxpolynomials.
	
	Here $(p \oplus q)(x)$ is not in full canonical form: $(p \oplus q)(x) = x^2 \oplus x \oplus 1$ has roots $r_1 = r_2 = \frac{1}{2}$. However, $(x \oplus \frac{1}{2})^2 = x^2 \oplus \frac{1}{2} x \oplus 1 \neq (p \oplus q)(x)$ as a maxpolynomial expression.
\end{example}

\section{The characteristic maxpolynomials of a matrix}
\label{sec:charpoly}

\begin{definition}
\label{def:permanent}
Let $A \in \Rnnmax$ be an $n \times n$ matrix over $\Rmax$.
The \textbf{max-plus permanent} of $A$ is
$$
\perm(A) = \bigoplus_{\sigma \in \SG_n} a_{1 \sigma(1)} \cdots a_{n \sigma(n)},
$$
where $\SG_n$ is the group of permutations on $[n] =\{1, \ldots,n \}$.
For subsets $I,J\subseteq [n]$ of equal cardinality
the permanent of the submatrix $A_{I,J} = [a_{i,j}]_{i\in I, j\in J}$
is called the $(I,J)$-\emph{minor} of $A$. The \emph{principal minors}
are the minors corresponding to $I=J$.
\end{definition}

\begin{definition}
The \textbf{characteristic maxpolynomial} of $A$, defined in \cite{C-G83}, is
$$
\chi_A (x) = \perm (x I \oplus A),
$$
where $I$ is the max-plus identity matrix with all entries on the main diagonal being 0 and all off-diagonal entries being $\maxzero$.
This is a polynomial of degree $n$, say $\chi_A(x)=\bigoplus_{k=0}^n c_k x^k$,
with $c_k=\delta_{n-k}(A)$, $k = 0, \ldots, n-1$, 
where
\begin{equation}
  \label{eq:charcoeff}
  \begin{aligned}
  \delta_k(A) & = \bigoplus_{\substack{I\subseteq [n]\\ \abs{I}=k}} \perm(A_{I,I})\\
         &= \bigoplus_{i_1,i_2,\dots,i_k}\bigoplus_{\sigma\in\SG_k} a_{i_1i_{\sigma(1)}} a_{i_2i_{\sigma(2)}}\dotsm a_{i_ki_{\sigma(k)}}
  \end{aligned}
\end{equation}
is the maximal value of the principal minors of order $k$ of $A$.

\end{definition}

As shown in \cite{ABG01}, the (tropical) roots of the characteristic
maxpolynomial $\chi_A (x)$, which are called the \textbf{max-plus eigenvalues} of $A$, can be asymptotically computed from the eigenvalues of an associated parametrized classical matrix with exponential entries.  

\begin{definition}
We define the \textbf{full characteristic maxpolynomial} of $A$ to be
$$
\fullchi{A} (x) = \perm (x 0 \oplus A),
$$
where 0 is the matrix with zeroes in all its entries.
The difference between $\fullchi{A} (x)$ and $\chi_A (x)$ is that in $\fullchi{A} (x)$ the indeterminate $x$ appears in all entries instead of just on the diagonal.
It follows that
$$
\fullchi{A} (x) \geq \chi_A (x)
$$
as maxpolynomial functions.
\end{definition}

Again, this is a polynomial of degree $n$, say $\fullchi{A}(x)=\bigoplus_{k=0}^n d_k x^k$,
whose $k$-th coefficient $d_k=\eta_{n-k}(A)$, $k = 0, \ldots, n-1$, is the maximal value of all minors of order $n-k$ of $A$, that is,
\begin{equation}
  \label{eq:fullcharcoeff2}
\begin{aligned}
   \eta_k(A) &= \bigoplus_{\substack{I,J\subseteq [n]\\ \abs{I}=\abs{J}=k}} \perm(A_{I,J})  \\
         &= \bigoplus_{\substack{i_1,i_2,\dots,i_k\\ j_1,j_2,\dots,j_k}} a_{i_1j_1} a_{i_2j_2}\dotsm a_{i_kj_k},
  \end{aligned}
\end{equation}
where the last sum in \eqref{eq:fullcharcoeff2} runs over all pairs of ordered
$k$-tuples of distinct indices. The relation to the principal minors is
$$
\delta_k(A)\leq \eta_k(A)=\bigoplus_P \delta_k(AP)
$$
where the last sum runs over all permutation matrices.

Similarly to eigenvalues, the (tropical) roots of the full characteristic maxpolynomial $\fullchi{A} (x)$, which are called the max-plus \textbf{singular values} of $A$, can be asymptotically computed from the singular values of an associated parametrized classical matrix with exponential entries (see \cite{Hook14} and \cite{DeSDeM02}).

The characteristic maxpolynomial is not necessarily in full canonical form.
However, this is the case for the full characteristic maxpolynomial, which
makes it a better choice for the investigation of convolutions of maxpolynomials in following sections.
\begin{theorem}
	The full characteristic maxpolynomial is in FCF.
	\label{thm:fullsemi}
\end{theorem}
\begin{proof}
 Let $\fullchi{A} (x) = \bigoplus_{k=0}^{n} d_k x^k$, where $d_n = 0$, be the full characteristic maxpolynomial of the $n \times n$ matrix $A$.
 The $(n-k)$th coefficient of $\fullchi{A} (x)$, $d_{n-k}=\eta_k$, $k = 1, \ldots, n$, equals the maximal permanent over all submatrices of $A$ of order $k$, which is the optimal value of the corresponding $k$-cardinality assignment problem \cite{BDM12}.
 As is shown below, the values $\eta_k$ form a concave sequence, that is, the inequalities $\eta_{k+1} - \eta_k \geq \eta_{k+2} - \eta_{k+1}$, $k = 0, \ldots, n-2$ hold.
 By Proposition~\ref{prop:fullcanonical} the full characteristic maxpolynomial $\fullchi{A} (x)$ is in full canonical form.
  
Indeed, the $k$-cardinality assignment problem can be formulated as an integer programming (IP) problem over the variables $x_{i,j}$, $1 \leq i,j \leq n$ as follows  \cite{DellamicoMartello:1997:assignment}:
$$
\sum_{i=1}^n \sum_{j=1}^n a_{ij} x_{ij} \to\max!
$$
subject to the conditions
\begin{subequations}
\begin{align}
  &\sum_{j=1}^n x_{ij} \leq 1 \qquad (i = 1, \ldots, n), \\
  &\sum_{i=1}^n x_{ij} \leq 1 \qquad (j = 1, \ldots, n), \\
  &\sum_{i=1}^n \sum_{j=1}^n x_{ij} = k, \\
  &x_{ij} \in\{0,1\} \qquad (i = 1, \ldots, n, \quad j = 1, \ldots, n).
  \label{eq:xij-int}
\end{align} 
\end{subequations}

If we replace the discrete condition  \eqref{eq:xij-int} by the linear condition
\begin{equation}
    \label{eq:xij-lin}
\tag{\ref{eq:xij-int}'}
x_{ij} \geq 0 \qquad (i = 1, \ldots, n, \quad j = 1, \ldots, n),
\end{equation}
we obtain the corresponding relaxed linear programming (LP) problem.

In general, there is no reason for the solution of the (LP) to
be integer valued and to solve the 
underlying (IP) problem.
This is however the case if the system matrix
is \emph{totally unimodular}, 
see \cite[Lemma~8.2.4]{MatousekGaertner:2007:understanding} or
\cite[Chapter~19]{Schrijver:1986:theory}.
The $k$-assignment problem satisfies this condition
\cite[Theorem~1]{DellamicoMartello:1997:assignment}.
It is a well known fact in the theory of linear programming
that the solution of the LP minimization problem is a convex function
of the constraint vector
\cite[Theorem~5.1]{BertsimasTsitsiklis:1997:introduction}
and it
follows that the sequence $\eta_k$  of solutions of the our
max problem is concave.

It is also beneficial to obtain the same geometric description of the solution set by representing the assignment problem in the language of a network flow problem \cite{DT97}.  
 \end{proof}
\begin{remark}
	The analogy between real-rootedness in standard algebra and the FCF property in max-plus algebra is demonstrated with regard to the characteristic polynomial.
	Given a real square matrix, its (standard) characteristic polynomial need not be real-rooted and likewise its max-plus counterpart need not be in FCF.
	However, as shown in Theorem~\ref{thm:fullsemi}, the full characteristic maxpolynomial is always in FCF.
	Similarly, the full characteristic polynomial $\det (xJ-A)$, where $J$ is the matrix whose entries are all equal to 1, is of degree 1 or 0 and thus always real-rooted.
\end{remark}
\subsection{The Gram characteristic maxpolynomial}
We start with a few definitions.
\begin{definition}
Given an $n \times n$ matrix $A \in \Rnnmax$ we call $G=A^TA$ the \textbf{max-plus Gram matrix} of $A$ and $\perm(G)$ the max-plus \textbf{Gram permanent} of $A$.
\end{definition}
Given two (column) vectors $u, v$ of size $n$, let $\langle u, v\rangle = u^T v$ be their max-plus scalar product.
The max-plus norm of $v$ is $\norm{v} \, = \langle v, v\rangle^{\half}$,
the maximal element of $v$.
Then the max-plus version of the Cauchy-Bunyakovsky-Schwarz inequality holds:
\begin{equation}
\label{eq:cauchy}
\langle u, v\rangle  \, \leq \, \norm{u} \norm{v}
\end{equation}
with equality if and only if the maxima of $u$ and $v$ occur at the same index,
i.e., if there is $1\leq i\leq n$ such that $\norm{u}=u_i$ and $\norm{v}=v_i$.

In the following we denote 
by $a_i$ the
column vectors of $A$.

\begin{proposition}
	Let $A=(a_{ij}) \in \Rnnmax$ and let $G=(g_{ij})=A^TA$ be its Gram matrix.
	Then the Gram permanent of $A$ is
	$$
	\perm(G) = 
	 \, \norm{a_1}^2\, \norm{a_2}^2 \dotsm \norm{a_n}^2
	 = g_{11} g_{22}\cdots g_{nn}.
         $$
	\label{prop:gperm}
\end{proposition}
\begin{proof}
	The entries of the matrix $G$ are
	\begin{equation*}
	g_{ij} = \langle a_i, a_j\rangle
	\end{equation*}
	and 
	\begin{equation*}
	\perm(G) = \bigoplus_{\pi \in \SG_n} \langle a_1, a_{\pi(1)} \rangle\langle a_2, a_{\pi(2)} \rangle\dotsm\langle a_n, a_{\pi(n)} \rangle
	.
	\end{equation*}
	Now, for a fixed permutation $\pi$ we apply the inequality \eqref{eq:cauchy} and obtain
	\begin{align*}
	\langle a_1, a_{\pi(1)} \rangle\langle a_2, a_{\pi(2)} \rangle\dotsm\langle    a_n, a_{\pi(n)} \rangle
	&\leq \norm{a_1}\norm{a_{\pi(1)}}
	\norm{a_2}\norm{a_{\pi(2)}}
	\dotsm
	\norm{a_n}\norm{a_{\pi(n)}}
	\\
	&= \norm{a_1}^2\,
	\norm{a_2}^2
	\dotsm
	\norm{a_n}^2\\
	&= g_{11}g_{22}\dotsm g_{nn},
	\end{align*}
	i.e., the maximal value is attained for $\pi=1$.
\end{proof}

Given a matrix $M$, we denote by $M^{\circ \half}$ the Hadamard root of $M$, i.e., the result of multiplying (in standard arithmetic) $M$ by the scalar $\half$.
For more on Hadamard product and Hadamard power of matrices see Subsection~\ref{subsec:HPmatrices}.

\begin{definition}
We denote by $\widehat{A}$ the Hadamard root of the Gram matrix of $A$,
$$
\widehat{A} = \hhalf{G} =\hhalf{(A^TA)},
$$ and call $\chi_{\widehat{A}} (x)$  the \textbf{Gram characteristic maxpolynomial} of $A$.
\end{definition}
The same arguments as in the proof of Proposition~\ref{prop:gperm} together with (\ref{eq:charcoeff}) show that we only need to consider the diagonal elements of $\widehat{A}$, that is, $\chi_{\widehat{A}} (x)$ is the FCF maxpolynomial
\begin{equation}
\label{eq:maxrow}
\begin{aligned}
\chi_{\widehat{A}} (x) &= (x \oplus \widehat{A}_{11})(x \oplus \widehat{A}_{22}) \cdots  (x \oplus \widehat{A}_{nn}) \\
&= (x \oplus m_1)(x \oplus m_2) \cdots  (x \oplus m_n),
\end{aligned}
\end{equation}
where $m_i = \norm{a_i}$ is the maximal element of column $i$ of $A$.
If follows that
$$
\chi_{\widehat{A}} (x) \geq \fullchi{A} (x)
$$
as maxpolynomial functions.

\section{Max convolution}
\label{sec:maxconv}
\subsection{Max convolution of maxpolynomials}
\begin{definition}
	Given two maxpolynomials $p(x)$, $q(x)$, their (formal) \textbf{max convolution} of order $k$ (or $k$-th max convolution) is the $k$-th derivative of their max-plus product: 
$$
(p \maxconv{k} q)(x) = (pq)^{(k)}(x).
$$
\end{definition}
If $p(x) = \bigoplus_{i=0}^{m} a_i x^i$, $q(x) = \bigoplus_{i=0}^{n} b_i x^i$ then
\begin{align*}
	(p \maxconv{k} q)(x)
	&= ( \bigoplus_{l=0}^{m+n} (\bigoplus_{i+j=l} a_i b_j) x^l )^{(k)} \\
	&= \bigoplus_{l=0}^{m+n-k} (\bigoplus_{i+j=l+k} a_i b_j) x^l.
\end{align*}
By Leibniz's rule (Proposition \ref{prop:leibniz}),
$$
(p \maxconv{k} q)(x) = \bigoplus_{i=0}^{k} p^{(i)}(x) \,q^{(k-i)}(x).
$$

For example, if $q(x) = 0$, the zero maxpolynomial, then $(p \maxconv{k} 0)(x)$ is the $k$-th derivative of $p(x)$, and if $k = 0$ then $(p \maxconv{0} q)(x)$ is the product of $p(x)$ and $q(x)$.
\begin{remark}
  We defined max convolution is on a  formal level since it relies on the formal derivative. A version respecting functional equivalence would use the functional version of the derivative, i.e.,
 we would first have to replace both $p(x)$ and $q(x)$ by their FCF representatives.
 Since our results below refer to maxpolynomials which are already in FCF, both definitions of max convolution coincide.
\end{remark}
\begin{proposition}
  For any FCF maxpolynomials $p(x), q(x)$ of degree $m, n$ respectively the following holds.
        \begin{enumerate}[(i)]
         \item 
          For every $k < m+n$ the max-convolution polynomial $(p \maxconv{k} q)(x)$ is an FCF maxpolynomial of degree $m+n-k$.
         \item 
          The roots of $(p \maxconv{k} q)(x)$ are the maximal $m+n-k$ roots (including multiplicities) among the roots of $p(x)$ and the roots of $q(x)$.
         \item
         In particular, when $p(x) = \bigodot_{i=1}^n (x \oplus r_i)$ and $q(x) = \bigodot_{i=1}^n (x \oplus s_i)$ then
          \begin{equation}
            (p \maxconv{n} q)(x) = \bigoplus_{\sigma \in \SG_n} \bigodot_{i=1}^n (x \oplus r_i \oplus s_{\sigma(i)}).
            \label{eq:maxroots}
          \end{equation}
        \end{enumerate}
	\label{prop:maxroots}
	\end{proposition}
\begin{proof}
    The product $(pq)(x)$ is a maxpolynomial of degree $m+n$ which is an FCF maxpolynomial by Proposition~\ref{prop:qe}.
    Its set of roots is the union of the roots of $p(x)$ and the roots of $q(x)$ (counting multiplicity).
	The $k$-th max convolution, i.e., the $k$-th derivative $(p \maxconv{k}q)(x) = (pq)^{(k)}(x)$ is again an FCF maxpolynomial by Proposition~\ref{prop:qe} and has degree $m+n-k$.
	By iterated application of Proposition~\ref{prop:rootsderivative} it follows that its roots are the $m+n-k$ maximal roots of $(pq)(x)$.
	
	As for \eqref{eq:maxroots}, it is clear that there exists a permutation $\sigma_0$ of $[n]$, such that $r_1 \oplus s_{\sigma_0(1)}$, $r_2 \oplus s_{\sigma_0(2)}, \ldots$, $r_n \oplus s_{\sigma_0(n)}$ are the $n$ maximal numbers among $r_1, \ldots, r_n$, $s_1, \ldots, s_n$.
	Since each coefficient of $\bigodot_{i=1}^n (x \oplus r_i \oplus s_{\sigma_0(i)})$ is greater than or equal to the corresponding coefficient of any other maxpolynomial $\bigodot_{i=1}^n (x \oplus r_i \oplus s_{\sigma(i)})$, the result follows.	
\end{proof}

The following properties of the max convolution are easily verified. 
\begin{proposition}
Let $p(x), p_1(x), p_2(x), q(x)$ be maxpolynomials. Then 
	\begin{enumerate}[(i)]
		\item Commutativity: $(p \maxconv{k} q)(x) = (q \maxconv{k} p)(x)$.
		\item Distributivity: $((p_1 \oplus p_2) \maxconv{k} q)(x) = (p_1 \maxconv{k} q)(x) \oplus (p_2 \maxconv{k} q)(x)$.
		\item Homogeneity: $(p_1 p_2 \maxconv{k} q)(x) = (p_1 \maxconv{k} p_2 q) (x)$.
		\item Leibniz's rule: $(p \maxconv{k} q)'(x) = (p \maxconv{k+1} q)(x) = (p' \maxconv{k} q)(x) \oplus (p \maxconv{k} q')(x)$.
	\end{enumerate}
	\label{prop:maxconvprop}
\end{proposition}
Associativity does not hold in general, but it is satisfied under the following
condition,  including the case where all maxpolynomials are of degree $n$ and the max convolution is of order $n$.
\begin{proposition}[Associativity]
	Let $p_1(x), p_2(x), p_3(x)$ be maxpolynomials of degrees $n_1, n_2, n_3 \leq n$, respectively.
	Then
	\begin{enumerate}[(i)]
		\item $((p_1 \maxconv{n} p_2) \maxconv{n} p_3)(x) = (p_1 \maxconv{n} (p_2 \maxconv{n} p_3))(x)$.
		\item If $p_1(x)$, $p_2(x)$, $p_3(x)$ are FCF maxpolynomials then the roots of $(p_1 \maxconv{n} p_2 \maxconv{n} p_3)(x)$ are the maximal $n_1+n_2+n_3 - 2n$ roots (including multiplicities) among the roots of $p_1(x)$, $p_2(x)$, $p_3(x)$.
	\end{enumerate}
	\label{prop:assoc}
\end{proposition}
\begin{proof}
  \begin{enumerate}[(i)]
   \item 
  We distinguish two cases.
  First assume that $n_1+n_2+n_3 < 2n$. 
  Then $((p_1 \maxconv{n} p_2) \maxconv{n} p_3)(x) = \maxzero$ for the following reason:
	either $n_1+n_2 < n$ and then $(p_1 \maxconv{n} p_2)(x)$ is formed by
        taking the $n$-th derivative of a maxpolynomial of degree $n_1+n_2 <
        n$, which gives $\maxzero$, and consequently 
        $((p_1 \maxconv{n} p_2) \maxconv{n} p_3)(x) = (\maxzero \maxconv{n} p_3)(x) = \maxzero$;
	otherwise $\deg (p_1 \maxconv{n} p_2) = n_1+n_2-n \geq 0$, and then in
        $((p_1 \maxconv{n} p_2) \maxconv{n} p_3)(x)$ 
        we take the derivative of order $n$ of a maxpolynomial of degree
        $(n_1+n_2-n) +n_3 < n$, which again results in $\maxzero$.
	Similarly, $(p_1 \maxconv{n} (p_2 \maxconv{n} p_3))(x) = \maxzero$, and associativity holds.
	
	Assume now that $2n \leq n_1+n_2+n_3 \leq 3n$. Then necessarily $n_1+n_2 \geq n$.
	Let $p_1(x) = \bigoplus_{i=0}^{n_1} a_i x^i$,  $p_2(x) = \bigoplus_{i=0}^{n_2} b_i x^i$ and $p_3(x) = \bigoplus_{i=0}^{n_3} c_i x^i$.
	The coefficients of $(p_1 \maxconv{n} p_2)(x) = (p_1 p_2)^{(n)}(x)$ are the leading $n_1+n_2+1-n$ coefficients of $(p_1 p_2)(x)$.
	That is,
	$$
	(p_1 \maxconv{n} p_2)(x) = \bigoplus_{i=0}^{n_1 + n_2 -n} d_i x^i = \bigoplus_{i=0}^{n_1 + n_2 -n} \Bigl(\smashoperator{\bigoplus_{\substack{j,k \\ j+k=n+i}}} a_j b_k \Bigr) x^i.
	$$
	Then the coefficients of $((p_1 \maxconv{n} p_2) \maxconv{n} p_3)(x)$ are the $n_1+n_2+n_3+1-2n$ leading coefficients of $(p_1 p_2 p_3)(x)$:
        \begin{equation}
		\label{eq:assoc}	
	\begin{aligned}
		((p_1 \maxconv{n} p_2) \maxconv{n} p_3)(x) &= \bigoplus_{i=0}^{n_1+n_2+n_3-2n} \Bigl(\smashoperator{\bigoplus_{\substack{m,l \\ m+l=n+i}}} d_m c_l \Bigr) x^i 
		 \\
		&= \bigoplus_{i=0}^{n_1+n_2+n_3-2n} \Bigl(\smashoperator[l]{\bigoplus_{\substack{m,l \\ m+l=n+i}}} \phantom{x}\Bigl(\smashoperator{\bigoplus_{\substack{j,k \\ j+k=n+m}}} a_j b_k\Bigr) c_l \Bigr) x^i 		 \\
		&= \bigoplus_{i=0}^{n_1+n_2+n_3-2n} \Bigl(\smashoperator{\bigoplus_{\substack{j,k,l \\ j+k=n+(n+i-l)}}} a_j b_k c_l \Bigr) x^i	\\
		&= \bigoplus_{i=0}^{n_1+n_2+n_3-2n} \Bigl(\smashoperator{\bigoplus_{\substack{j,k,l \\ j+k+l=2n+i}}} a_j b_k c_l \Bigr) x^i.
	\end{aligned}
        \end{equation}
	
	By the symmetry of the last expression of \eqref{eq:assoc} and by commutativity, the order in which
        the 3 maxpolynomials are max convolved is irrelevant and associativity
        follows. 
       \item 
	This is clear by Proposition~\ref{prop:maxroots}.
  \end{enumerate}
\end{proof}

We remark that when $A_i$, $i =0, \ldots, k$, are $n\times n$ matrices over $\Rmax$, with full characteristic maxpolynomials $p_i(x)$, $i =0, \ldots, k$, then
$$
(p_0 \maxconv{n} p_1 \maxconv{n} \cdots \maxconv{n} p_k)(x) = \bigoplus_{\substack{P_i, Q_i \in \cP_n \\ i =1, \ldots, k}}  \mathrm{perm} (x 0 \oplus A_0 \oplus P_1 A_1 Q_1 \oplus \cdots \oplus P_k A_k Q_k),
$$
whose roots are the maximal $n$ roots (including multiplicities) among the roots of $p_i(x)$, $i =0, \ldots, k$.
This follows from Theorem \ref{thm:maxconv} below.
\subsection{Max convolution of characteristic maxpolynomials}
Let $\cP_n$ be the group of $n \times n$ max-plus permutation matrices, which are similar to the standard permutation matrices, except that the entries are assigned the values $0$ and $\maxzero$ instead of $1$ and $0$, respectively.
The permutation matrices are the orthogonal matrices in the max-plus setting. 
Given two matrices $A, B \in \Rnnmax$, we show next that the $n$-th max convolution of their full characteristic maxpolynomials equals the maximum ($\oplus$), over all permutation matrices $P, Q \in \cP_n$, of the set of full characteristic maxpolynomials of $A \oplus PBQ$.
Recall from \eqref{eq:fullcharcoeff2} that by $\eta_k(A)$ we denote the maximal
value of all minors of order $k$ of a matrix $A$.
\begin{theorem}
  \label{thm:maxconv}
  Given matrices $A, B$ of order $n$ over $\Rmax$,
  let $p(x)=\fullchi{A}(x)$, $q(x)=\fullchi{B}(x)$ be the corresponding full characteristic maxpolynomials and let \\
$(p \maxconv{n} q)(x) =  \sum_{k=0}^n d_kx^k$ be their max convolution.
  \begin{enumerate}[(i)]
   \item The max convolution can be written as
        \begin{equation}
          \label{eq:fullcharconvolution}
		(p \maxconv{n} q)(x) = \bigoplus_{P,Q \in \cP_n}  \fullchi{A \oplus PBQ}(x).
        \end{equation}
       \item The coefficients of the max convolution are given by
        \begin{equation}
          \label{eq:maxconv:ck}
          d_{n-k} = \bigoplus_{l=0}^k \eta_l(A)\,\eta_{k-l}(B).
        \end{equation}
       \item
        \label{it3:thm:maxconv}
        The roots of $(p \maxconv{n} q)(x)$ are the maximal $n$ roots
        (including multiplicities) among the roots of $p(x)$ and $q(x)$.
  \end{enumerate}
\end{theorem}
\begin{proof}
  From \eqref{eq:fullcharcoeff2} we infer that the coefficients $c_k$ are certain maximal minors.
  To be specific, 
  \begin{align*}
    c_{n-k} &= \bigoplus_{P,Q} \eta_{k}(A\oplus PBQ)\\
           &= \bigoplus_{\sigma,\tau\in\SG_n}
              \bigoplus_{\substack{i_1,i_2,\dots,i_k\\ j_1,j_2,\dots,j_k}}
              (a_{i_1j_1} \oplus b_{\sigma(i_1)\tau(j_1)})
              (a_{i_2j_2} \oplus b_{\sigma(i_2)\tau(j_2)})    
              \dotsm
              (a_{i_kj_k} \oplus b_{\sigma(i_k)\tau(j_k)}),
    \intertext{
    where $\sigma$ and $\tau$ are the inverses of the permutations induced by the respective permutation
    matrices $P$ and $Q$, 
    and where both tuples $(i_1,i_2,\dots,i_k)$ and $(j_1,j_2,\dots,j_k)$ consist of distinct indices.
    Using the latter fact, we can expand the products, regroup and relabel the indices to obtain
}
           &= \bigoplus_{\sigma,\tau\in\SG_n}
              \bigoplus_{l=0}^k
              \bigoplus_{\substack{i_1,i_2,\dots,i_k\\
                                   j_1,j_2,\dots,j_k}}
              a_{i_1j_1}
              a_{i_2j_2}
              \dotsm
              a_{i_lj_l}
b_{\sigma(i_{l+1})\tau(j_{l+1})}
b_{\sigma(i_{l+2})\tau(j_{l+2})}
\dotsm
b_{\sigma(i_k)\tau(j_k)}
              ;
\intertext{now $\sigma$ and $\tau$ are arbitrary permutations and after
    removing duplicated summands we remain with}
           &= \bigoplus_{l=0}^k
              \bigoplus_{\substack{i_1,i_2,\dots,i_l\\
                                   j_1,j_2,\dots,j_l\\
                                   i'_1,i'_2,\dots,i'_{k-l}\\
                                   j'_1,j'_2,\dots,j'_{k-l}}}
              a_{i_1j_1}
              a_{i_2j_2}
              \dotsm
              a_{i_lj_l}
              b_{i'_1j'_1}
              b_{i'_2j'_2}
              \dotsm
              b_{i'_{k-l}j'_{k-l}}
    ;
\intertext{in this sum the entries of $A$ and $B$ are decoupled and again by
    \eqref{eq:fullcharcoeff2} it is further equal to}
         &=\bigoplus_{l=0}^k \eta_l(A)\, \eta_{k-l}(B)\\
         &=\bigoplus_{l=0}^k a_{n-l} b_{n-(k-l)}\\    
         &= \bigoplus_{i+j=2n-k} a_ib_j,
  \end{align*}
  which is indeed $d_{n-k}$, the $(n-k)$th coefficient of $(p \maxconv{n} q)(x) = (pq)^{(n)}(x)$.
  This completes the proof of (i) and (ii).
  
  So far, we did not make use of the full canonical form.
  It is, however, essential for item (iii)
  and the discussion of the roots of $(p \maxconv{n} q)(x)$
  First, we observe that both $p(x)$ and $q(x)$ are FCF maxpolynomials by Theorem~\ref{thm:fullsemi}.
  Now, we infer from Proposition~\ref{eq:maxroots} that the convolution 
  $(p \maxconv{n} q)(x)$ has the same property and by Proposition~\ref{prop:maxroots} its roots are the maximal $n$ roots among the roots of $p(x)$ and the roots of $q(x)$.
\end{proof}

\subsubsection{Principally dominant matrices}
\begin{definition}
  A square matrix $A \in \Rnnmax$ is \textbf{max-plus principally dominant} if for every $k \in\{ 1, \ldots, n\}$ the maximal minor permanent of order $k$ is achieved on a principal submatrix of $A$ of order $k$, see Definition~\ref{def:permanent}.
\end{definition}

For example, diagonal matrices are principally dominant, as well as diagonally dominant
matrices (the diagonal elements are the maximal elements of their rows).
Also Gram matrices are  principally dominant.
In fact, it is easy to generalize Proposition~\ref{prop:gperm} to minors and to show that for a Gram matrix $G=A^TA$ the maximal minor is
\begin{equation}
  \label{eq:grammaxminor}
  \eta_k(A^TA) = \gamma^\downarrow_1\gamma^\downarrow_2\dotsm \gamma^\downarrow_k ,
\end{equation}
where the vector $(\gamma^\downarrow_1,\gamma^\downarrow_2,\dots,\gamma^\downarrow_n)$ is
the \emph{nonincreasing rearrangement} of the vector $(g_{11},g_{22},\dots,g_{nn})= (\norm{a_1}^2,\norm{a_2}^2,\dots,\norm{a_n}^2)$ of squared norms of the columns of the matrix $A$.
\begin{remark}
	A matrix may be symmetric and not principally dominant, e.g.,
	$
	\begin{bmatrix}
	0 & 1 \\
	1 & 0 
	\end{bmatrix},
	$
	or principally dominant and not symmetric, e.g.,
	$
	\begin{bmatrix}
	2 & 1 \\
	0 & 0 
	\end{bmatrix}.
	$
\end{remark}
\begin{remark}
	When $A$ and $B$ are principally dominant then $C = A \oplus B$ is not necessarily principally dominant. For example, in
	$$
	\begin{bmatrix}
	6 & 5 & 0 & 0 \\
	5 & 0 & 3 & 0 \\
	0 & 2 & 0 & 0 \\
	0 & 0 & 0 & 0
	\end{bmatrix}
	\oplus
	\begin{bmatrix}
	6 & 5 & 0 & 0 \\
	5 & 0 & 0 & 2 \\
	0 & 0 & 0 & 0 \\
	0 & 3 & 0 & 0
	\end{bmatrix}
	=
	\begin{bmatrix}
	6 & 5 & 0 & 0 \\
	5 & 0 & 3 & 2 \\
	0 & 2 & 0 & 0 \\
	0 & 3 & 0 & 0
	\end{bmatrix}
	$$
	both $A$ and $B$ are principally dominant, but the maximal minor permanent of order $3$ in $C=A \oplus B$, which is $6 \odot 3 \odot 3 = 12$, is not achieved on any principal submatrix. 
\end{remark}
The next proposition follows immediately from \eqref{eq:charcoeff} and \eqref{eq:fullcharcoeff2}.
\begin{proposition}
  A matrix $A$ is principally dominant if and only if $\chi_A (x) = \fullchi{A} (x)$.
\end{proposition}

In general, $(p \maxconv{n} q)(x) \geq \bigoplus_{P \in \cP_n}  \chi_{A \oplus PBP^T}(x)$ when $p(x)=\chi_{A}(x)$ and $q(x)=\chi_{B}(x)$.
However, when $A$ and $B$ are principally dominant then equality holds and we have the following version of Theorem~\ref{thm:maxconv}, which shows that in this case the max convolution can be computed on a set of $n!$ instead of $(n!)^2$ permutation matrices.
Recall from \eqref{eq:charcoeff} that by $\delta_k(A)$ we denote the maximal
value of all principal minors of order $k$ of a matrix $A$.
\begin{theorem}
	\label{thm:pd_maxconv}
	Given principally dominant matrices $A, B$ of order $n$ over $\Rmax$, let $p(x)=\chi_{A}(x)$, $q(x)=\chi_{B}(x)$ be the corresponding characteristic maxpolynomials and let $(p \maxconv{n} q)(x) = \sum_{k=0}^n d_kx^k$ be their max convolution.
	\begin{enumerate}[(i)]
         \item \label{it1:thm:pd-maxconv}
          The max convolution can be written as
          \begin{equation}
            \label{eq:pd_convolution}
            (p \maxconv{n} q)(x) = \bigoplus_{P \in \cP_n}  \chi_{A \oplus PBP^T}(x).
          \end{equation}
         \item \label{it2:thm:pd-maxconv}
          The coefficients of the max convolution evaluate to
          \begin{equation}
            \label{eq:pd_maxconv:ck}
            d_{n-k} = \bigoplus_{l=0}^k \delta_l(A)\,\delta_{k-l}(B).
          \end{equation}
	\end{enumerate}
\end{theorem}
\begin{proof}
  Since $A, B$ are principally dominant we have $p(x)=\chi_A(x) = \fullchi{A}(x)$ and $q(x)=\chi_B(x)=\fullchi{B}(x)$.
  From \eqref{eq:maxconv:ck} we infer that the coefficients on the left-hand side of \eqref{eq:pd_convolution} are
  \begin{align*}
    d_{n-k}
    &=\bigoplus_{l=0}^k \eta_l(A)\,\eta_{k-l}(B)
      \intertext{and since both $A$ and $B$ are principally
      dominant, we can replace the minors by principal minors and obtain}
    &=\bigoplus_{l=0}^k \delta_l(A)\,\delta_{k-l}(B)
      .
  \end{align*}
  We  have thus proved \ref{it2:thm:pd-maxconv}.
  To prove \ref{it1:thm:pd-maxconv} it remains to show that the coefficients $c_k$ 
  of the maxpolynomial $\bigoplus_{P \in \cP_n}  \chi_{A \oplus PBP^T}(x) =
  \sum_{k=0}^n c_k x^k$ on the right-hand side of \eqref{eq:pd_convolution}
  coincide
  with those of the maxpolynomial $\bigoplus_{P,Q \in \cP_n}  \fullchi{A \oplus
    PBQ}(x) = \sum_{k=0}^n d_k x^k$ which we just computed~\eqref{eq:pd_maxconv:ck}.

  The idea is as follows.
  The coefficient $c_{n-k}$ equals the (standard) sum of a maximal permanent of some submatrix $A'$ of $A$ of order $l$ and a maximal permanent of a submatrix $B'$ of $B$ of order $k-l$, where there is no common row index or common column index between $A'$ and the image of $B'$ in $PBQ$.
  But in the principal dominant case, since the maximizing submatrices
  can be chosen to be principal, the permutation matrix $Q$ may be chosen
  to be equal to $P^T$ and thus $A'$ and the image of $B'$ in $PBP^T$ are decoupled.

  The details of the calculation are as follows.
  Applying \eqref{eq:charcoeff} we can write
  \begin{align*}
    c_{n-k}
    &= \bigoplus_{P\in \cP_n} \delta_k(A\oplus PBP^T)\\
    &=
      \begin{multlined}[t]
        \bigoplus_{\pi\in\SG_n}
        \bigoplus_{i_1,i_2,\dots,i_k}\bigoplus_{\sigma\in\SG_k}
        (a_{i_1i_{\sigma(1)}}\oplus b_{\pi(i_1)\pi(i_{\sigma(1)})})
        (a_{i_2i_{\sigma(2)}}\oplus b_{\pi(i_2)\pi(i_{\sigma(2)})})
        \\
        \dotsm
        (a_{i_ki_{\sigma(k)}}\oplus b_{\pi(i_k)\pi(i_{\sigma(k)})}),
      \end{multlined}
    \intertext{
    where we switched the notation from the permutation matrices $P$
    to the corresponding permutations $\pi$. In order to keep the
    proliferation of indices within manageable bounds, we now replace the sequences of
    distinct indices by injective functions
    $g,h:[k]\to[n]$, where $h=\pi\circ g$, and obtain
    }
    &=
      \begin{multlined}[t]
        \bigoplus_{g,h:[k]\to[n]}
        \bigoplus_{\sigma\in\SG_k}
        (a_{g(1)g(\sigma(1))}\oplus b_{h(1)h(\sigma(1)))})
        (a_{g(2)g(\sigma(2))}\oplus b_{h(2)h(\sigma(2)))})
        \\
        \dotsm
        (a_{g(k)g(\sigma(k))}\oplus b_{h(k)h(\sigma(k)))})
        ;
      \end{multlined}
    \intertext{
    after expanding the product we obtain a sum over all partitions
    of $[k]$ into two subsets 
    which we denote     by index sequences    $\underline{i}$ and
    $\underline{j}$     of size $l$ and $k-l$, respectively:
    }
    &=
      \begin{multlined}[t]
        \bigoplus_{g,h:[k]\to[n]}
        \bigoplus_{l=0}^k
        \bigoplus_{\sigma\in\SG_k}
        a_{g(i_1)g(\sigma(i_1))}
        a_{g(i_2)g(\sigma(i_2))}
        \dotsm
        a_{g(i_l)g(\sigma(i_l))}
        \phantom{xxxxx}
        \\
        b_{h(j_1)h(\sigma(j_1))}
        b_{h(j_2)h(\sigma(j_2))}
        \dotsm
        b_{h(j_{k-l})h(\sigma(j_{k-l}))}
        ;
      \end{multlined}
    \intertext{
    now the entries of $A$ and $B$ are decoupled
    and since both $A$ and $B$ are principally dominant the maximal values are
    attained when both $\underline{i}$ and $\underline{j}$ are invariant under
    $\sigma$ and thus give rise to term from a principal minor, yielding
    }
    &=\bigoplus_{l=0}^k \delta_l(A)\,\delta_{k-l}(B)
  \end{align*}
  as claimed.
\end{proof}
\subsubsection{Symmetric matrices}
The convolution formulas of \cite{MSS15} are based on symmetric matrices.
In Theorem~\ref{thm:pd_maxconv} we got an analogous formula in the max-plus setting for matrices that are principally dominant.  
The following example shows that symmetry of the matrices is not
the right ingredient in our setting.
When a max-plus matrix $A$ is symmetric then $\chi_A (x)$ and $\fullchi{A} (x)$ have the same roots (see \cite{Hook14}), that is, $\chi_A (x)$ and $\fullchi{A} (x)$ induce the same polynomial  function, however, unlike the full characteristic polynomial $\fullchi{A} (x)$, the plain characteristic polynomial $\chi_A (x)$ is
not necessarily
in FCF.
\begin{example}
\label{ex:symmetric}
Let $A, B$ be the symmetric matrices
$$
A =
\begin{bmatrix*}[r]
2 & \maxzero \\
\maxzero & 0
\end{bmatrix*},
\qquad
B =
\begin{bmatrix*}[r]
0 & 10 \\
10 & 0
\end{bmatrix*}.
$$
Then $\chi_A (x) = x^2 \oplus 2x \oplus 2$ with roots $(2,0)$, while 
$\chi_B (x) = x^2 \oplus x \oplus 20$ with roots $(10,10)$.
For each $P \in \cP_2$ we get 
$$
A \oplus PBP^T = 
\begin{bmatrix*}[r]
2 & \maxzero \\
\maxzero & 0
\end{bmatrix*}
$$
and therefore $\bigoplus_{P \in \cP_2}  \mathrm{perm} (x I \oplus A \oplus PBP^T) = \chi_A (x) = x^2 \oplus 2x \oplus 2$ with roots $(2,0)$.
The max convolution of $\chi_A (x)$ and $\chi_B (x)$ is $x^2 \oplus 2x^3 \oplus 20$ with roots $(10,10)$.
Thus, in this case, where $B$ is not principally dominant, Equation~\eqref{eq:pd_convolution} of  Theorem~\ref{thm:pd_maxconv} does not hold.
We remark that equation~\eqref{eq:pd_convolution} does not hold here even functionally, nor would it help to use the functional convolution instead of the formal one. 

Let us now replace $B$ by the matrix 
$$
B' = \begin{bmatrix*}[r]
10 & 0 \\
0 & 10
\end{bmatrix*}.
$$
Then $\chi_{B'} (x) = x^2 \oplus 10x \oplus 20$ has the same roots as $\chi_B$, namely $(10,10)$.
However, for each $P \in \cP_2$, 
$$
A \oplus PB'P^T = 
\begin{bmatrix*}[r]
10 & \maxzero \\
\maxzero & 10
\end{bmatrix*}
$$
and therefore $\bigoplus_{P \in \cP_2}  \mathrm{perm} (x I \oplus A \oplus PB'P^T) = 
x^2 \oplus 10x \oplus 20 = (x \oplus 10)^2$,
which equals the max convolution of $\chi_A (x)$ and $\chi_{B'} (x)$.
In this case $A$ and $B'$ are principally dominant and Theorem~\ref{thm:pd_maxconv} applies.

Observe also that $\fullchi{A} (x)= \chi_A (x) = x^2 \oplus 2x \oplus 2 = (x \oplus 2) (x \oplus 0)$ and $\fullchi{B} (x)=x^2 \oplus 10x \oplus 20 = (x \oplus 10)^2$.
Then $\bigoplus_{P, Q \in \cP_2}  \mathrm{perm} (x 0 \oplus A \oplus PBQ) = x^2 \oplus 10x \oplus 20 = (\fullchi{A} \maxconv{2} \fullchi{B})(x)$ and so Theorem~\ref{thm:maxconv} applies.
\end{example}

\subsubsection{Max-row convolution}
In \cite{MSS15} the ``asymmetric additive convolution'' of the characteristic polynomials $p(x)$ and $q(x)$ of $AA^T$ and $BB^T$, respectively,
is defined as
$$
p \boxplus \! \! \boxplus_n \, q(x) = \mathbb{E}_{P,Q} \, \, \chi_{(A+PBQ)(A+PBQ)^T} (x),
$$
where the expectation is computed by randomly sampling the matrices $P,Q$ over the set of orthonormal matrices equipped with the Haar measure.
But if we look at the Gram characteristic polynomial of $(A \oplus PB)^T$, i.e., the characteristic polynomial of $((A \oplus PB)(A \oplus PB)^T))^{\circ \half}$ (note the Hadamard power of $\half$), then the max-plus analogue, the \textbf{max-row convolution}, can be expressed through the already defined max convolution, as shown below.
In the following theorem we denote by $M_i$ the $i$-th row of a matrix $M$.
\begin{theorem}
	Let $A, B$ be matrices of order $n$ over $\Rmax$, and
	let $m_i$, (resp.~$l_i$), $i = 1, \ldots, n$, be the maximal element of row $i$ in $A$ (resp.~$B$).
	Then the max convolution of the characteristic maxpolynomials 
        $p(x) = \chi_{\widehat{A^T}} (x)$ and $q(x) = \chi_{\widehat{B^T}} (x)$
        is
	\be
	(p \maxconv{n} q)(x) = \bigoplus_{P \in \cP_n} \, \, \chi_{\widehat{C(P)}} (x),
	\label{eq:maxrowconv}
	\ee
	where $C(P) = (A \oplus PB)^T$
	and the roots of $(p \maxconv{n} q)(x)$ are the maximal $n$ numbers among $(m_1, \ldots, m_n, l_1, \ldots, l_n)$.
	\end{theorem}
	\begin{proof}
		By \eqref{eq:maxrow}, the roots of $\chi_{\widehat{M}} (x)$ are the maximal elements of the columns of $M$.
		Thus, for a fixed permutation matrix $P$,
		$$
		\chi_{\widehat{C(P)}} (x) = (x \oplus r_1)(x \oplus r_2) \cdots (x \oplus r_n),
		$$
		where, for each $i$,
		$$
		r_i = \max_{1 \leq j \leq n} (A \oplus PB)^{T}_{ji} = \max_{1 \leq j \leq n} (A \oplus PB)_{ij} = \norm{ (A \oplus PB)_i} = \norm{A_i} \oplus \norm{(PB)_i}.
		$$
		Let $P$ represent the permutation $\pi \in \SG_n$. Then
		$$
		r_i = \norm{A_i} \oplus \norm{B_{\pi(i)}} = m_i \oplus l_{\pi(i)}
		$$
		and
		$$
		\chi_{\widehat{C(P)}} (x) = (x \oplus m_1 \oplus l_{\pi(1)}) (x \oplus m_2 \oplus l_{\pi(2)}) \cdots (x \oplus m_n \oplus l_{\pi(n)}).
		$$
		This is a maxpolynomial of degree $n$ with roots $r_i$, which are n elements among $(m_1, \ldots, m_n, l_1, \ldots, l_n)$.
		Clearly, the maximum over all these maxpolynomials is the one whose roots are the maximal $n$ elements among $(m_1, \ldots, m_n, l_1, \ldots, l_n)$.
		We claim that this maxpolynomial is achieved when going over all permutations $\pi$ in the right hand side of \eqref{eq:maxrowconv}.
		Indeed, suppose that the elements $m_i$ are arranged in decreasing order according to the permutation $\sigma$, that is, $m_{\sigma(1)} \geq m_{\sigma(2)} \geq \cdots \geq m_{\sigma(n)}$, and the elements $l_i$ are arranged in decreasing order according to the permutation $\tau$: $l_{\tau(1)} \geq l_{\tau(2)} \geq \cdots \geq l_{\tau(n)}$.
		Then the maximum is achieved for the permutation $\pi$ which couples $l_{\tau(i)}$ with $m_{\sigma(n+1-i)}$:
		$$
		\pi(i) = \sigma(n+1-\tau^{-1}(i)).
		$$
		For example, if $m_1 \geq m_2 \geq \cdots \geq m_n$ and $l_1 \geq l_2 \geq \cdots \geq l_n$ ($\sigma$ and $\tau$ are the identity permutation), and $m_1, \ldots, m_k$, $l_1, \ldots, l_{n-k}$ are the maximal $n$ elements among the $m_i$ and $l_i$ then for $\pi(i) = n+1-i$, $i=1,\ldots,n$, the roots of $\chi_{\widehat{C(P)}} (x)$, where $P$ is the permutation matrix representing $\pi$, are:
		\begin{align*}
		&m_1 \oplus l_n, m_2 \oplus l_{n-1}, \ldots, m_k \oplus l_{n+1-k},
		m_{k+1} \oplus l_{n-k}, \ldots, m_n \oplus l_1 \\
		= \, &m_1, m_2, \ldots, m_k, l_{n-k}, \ldots, l_1.
		\end{align*}
		These roots are also the roots of $(p \maxconv{n} q)(x)$, the left hand side of \eqref{eq:maxrowconv}, by the definitions of the Gram characteristic maxpolynomial and the max convolution.
\end{proof}

A max-column convolution can be defined in an analogous way.
\section{Hadamard product}
\label{sec:hadamard}
\subsection{Hadamard product of maxpolynomials}
Given two maxpolynomials $p(x) = \bigoplus_{i=0}^{n} a_i x^i$, $q(x) = \bigoplus_{i=0}^{n} b_i x^i$ of degree $n$, their max-plus \textbf{Hadamard product} is defined as 
$$
(p \circ q)(x) = \bigoplus_{i=0}^{n} a_i b_i x^i,
$$
that is, the coefficients of Hadamard product $(p \circ q)(x)$ are the max-products (standard sums) of the corresponding coefficients of $p(x)$ and $q(x)$.
In general, the roots of $(p \circ q)(x)$ are not the max-products of the corresponding roots of $p(x)$ and $q(x)$.
For example, let $p(x) = x^2 \oplus 4x \oplus 4$ and let $q(x) = x^2 \oplus 1x \oplus 3$.
The roots of $p(x)$ are $(4,0)$ and the roots of $q(x)$ are $(1.5,1.5)$, whereas
the roots of $(p \circ q)(x) = x^2 \oplus 5x \oplus 7$ are $(5,2)$. 
However, for FCF maxpolynomials we have the following result.
\begin{proposition}
	Let $p(x), q(x)$ be FCF maxpolynomials with roots $r_1 \geq \cdots \geq r_n$ and $s_1 \geq \cdots \geq s_n$, respectively.
	Then $(p \circ q)(x)$ is an FCF maxpolynomial with roots $t_i = r_i\odot s_i = r_i + s_i$, $i = 1, \ldots, n$. That is,
	\be
	(p \circ q)(x) = \bigoplus_{\sigma \in \SG_n} \bigodot_{i=1}^n (x \oplus r_i \odot s_{\sigma(i)}) = \bigodot_{i=1}^n (x \oplus r_i \odot s_i).
	\label{eq:hadamardroots}
	\ee
	\label{prop:hadamardroots}
\end{proposition}
\begin{proof}
	Let $p(x) = \bigoplus_{i=0}^{n} a_i x^i$, $q(x) = \bigoplus_{i=0}^{n} b_i x^i$ and let $(p \circ q)(x) = \bigoplus_{i=0}^{n} c_i x^i$, where $c_i = a_i b_i$, be their Hadamard product.
	Then
	\begin{align*}
	c_{i-1} - c_i &= (a_{i-1} + b_{i-1}) - (a_i + b_i) = (a_{i-1} - a_i) + (b_{i-1} - b_i) \\
	&\leq (a_i - a_{i+1}) + (b_i - b_{i+1}) = (a_i + b_i) - (a_{i+1} +  b_{i+1}) \\
	&= c_i - c_{i+1}
	\end{align*}
	for $i = 1, \ldots, n-1$.
	It follows that $(p \circ q)(x)$ is FCF.
	Moreover,
	$$
	t_i = c_{i-1} - c_i = (a_{i-1} - a_i) + (b_{i-1} - b_i) = r_i + s_i
	$$
	for $i = 1, \ldots, n$.
	
	As for \eqref{eq:hadamardroots}, it follows from the fact that the coefficient of $x^k$ in each $\bigodot_{i=1}^n (x \oplus r_i \odot s_{\sigma(i)})$ is a max-product of $n-k$ roots $r_i$ and $n-k$ roots $s_j$, and this term is maximal when $\sigma$ is the identity permutation. 
\end{proof}
The following properties of the Hadamard product of maxpolynomials
are easily verified. 
\begin{proposition}
	Let $p(x), p_1(x), p_2(x), q(x), q_1(x), q_2 (x)$ be maxpolynomials. Then 
	\begin{enumerate}[(i)]
		\item Commutativity: $(p \circ q)(x) = (q \circ p)(x)$.
		\item Associativity:  $ ((p_1 \circ p_2) \circ p_3 )(x) = (p_1 \circ ( p_2 \circ p_3) )(x)$. 
		\item Distributivity: $((p_1 \oplus p_2) \circ q)(x) = (p_1 \circ q)(x) \oplus (p_2 \circ q)(x)$.
		\item $((p_1 \circ q_1)(p_2 \circ q_2)) (x) \le ((p_1p_2  ) \circ (q_1 q_2))(x)$.
		\item $((p_1 \circ q_1)  \maxconv{k} (p_2 \circ q_2)) (x) \le ((p_1  \maxconv{k} p_2  ) \circ (q_1  \maxconv{k} q_2))(x)$.
	\end{enumerate}
	\label{prop:Hadamardprop}
\end{proposition}

\subsection{Hadamard product of matrices}
\label{subsec:HPmatrices}
The \textbf{max-plus Hadamard product of matrices} is the analogue of the standard Hadamard product in max-plus algebra.
That is, if $A,B$ are two $m \times n$ max-plus matrices then their Hadamard product is $C = A \circ B$, where $C$ is an $m \times n$ matrix satisfying
$$
c_{ij} = a_{ij} \odot b_{ij},
$$
i.e. $c_{ij} = a_{ij} + b_{ij}$ in standard arithmetic.
The Hadamard power $A^{\circ t}$, $t > 0$, of $A = (a_{ij})$ is then naturally defined: $(A^{\circ t})_{ij} = ta_{ij}$, where the product $ta_{ij}$ is the standard one. 

Below we list some properties of the Hadamard product and Hadamard powers. Let $\nu (A)$ denote the largest eigenvalue of $A$, i.e., the largest root of $\chi_A (x)$,  and let $\|A\| = \max \{a_{ij}: i=1,\ldots , m,\; j=1, \ldots , n  \}$.
As before, we denote by $\widehat{A}$ the Hadamard root of the Gram matrix of $A$, i.e., $\widehat{A} = \hhalf{G} =\hhalf{(A^TA)}$ and $\widehat{A^T}=\hhalf{(AA^T)}$.
The following properties are known or easy to prove (see, e.g., \cite{P08,P12}). 
\begin{proposition} Let $A,B, A_1, \ldots, A_m \in \Rnnmax$ and $t>0$. Then we have
	$$\nu (A\circ B) \le \nu (A)\, \nu (B), \; \|A\circ B\| \le \|A\| \|B\|,$$
	$$\nu (A^{\circ t}) = \nu (A)^t,\;  \|A^{\circ t}\| = \|A\|^t,  $$
	$$A_1 ^{\circ t} \cdots  A_m ^{\circ t} =(A_1 \cdots A_m)^{\circ t}, $$
	$$\nu (AB)=\nu (BA), $$
	$$\|\widehat{A}\| = \|\widehat{A^T}\|= \|A\| = \nu (\widehat{A}) = \nu (\widehat{A^T}),$$
	$$\nu (A_1 \circ \cdots \circ A_m) \le \nu (A_1 \cdots A_m).$$
\end{proposition}

Similarly to \cite{DP16}, we can prove the following max-plus version of 
\cite[Corollary 3.5, Theorem 3.9, Corollary 3.10]{DP16}.
In the proof it is useful to switch to the isomorphic max-times algebra setting by using the equality $\nu (A) = \log \mu (B)$, where $B$ denotes a non-negative $n\times n$ matrix $B=[e^{a_{ij}}]$ and $\mu (B)$ denotes the largest max-times eigenvalue of $B$.
Then the result follows by replacing the standard product of matrices by the max-times product and by applying the max-times  Gelfand formula for  $\mu (B)$ (see, e.g., \cite[Equality (4)]{P08,P12,MP15}) in the proofs of \cite[Corollary 3.5, Theorem 3.9, Corollary 3.10]{DP16}.
To avoid too much repetition of ideas from \cite{DP16} we omit the details of the proof.
\begin{theorem} Let $A,B, A_1, \ldots, A_m  \in \Rnnmax$. Then we have
	$$\|A\circ B\| \le \nu (A^TB),$$
	$$
	\|A_1 \circ A_2 \circ \cdots \circ A_m \| \le  \nu (\widehat{A_1} \circ \widehat{A_2} \circ \cdots \circ \widehat{A_m}) 
	 \le  \nu (\widehat{A_1} \widehat{A_2} \cdots \widehat{A_m}),
	$$
	$$
	\|A_1 \circ A_2 \circ \cdots \circ A_m \| \le  \nu (\widehat{A_1^T} \circ \widehat{A_2^T} \circ \cdots \circ \widehat{A_m^T})  
	\le  \nu (\widehat{A_1^T} \widehat{A_2^T} \cdots \widehat{A_m^T}).
	$$
	
	If $m$ is even then
	\begin{align*}
	 \|A_1 \circ A_2 \circ \cdots \circ A_m \|^2 &\le
	\nu (A_1 ^TA_2A_3 ^TA_4 \cdots A_{m-1} ^T A_m)\, \nu (A_1A_2 ^TA_3A_4 ^T \cdots A_{m-1} A_m ^T  ) \\
	&= \nu (A_1 ^TA_2A_3 ^TA_4 \cdots A_{m-1} ^T A_m) \, \nu (A_mA_{m-1} ^T \cdots A_{4} A_3 ^T A_2A_1 ^T ).
	\end{align*}

	If $m$ is odd then
	$$
	\|A_1 \circ A_2 \circ \cdots \circ A_m \|^2 \le  \nu (A_1 A_2 ^TA_3A_4 ^T \cdots A_{m-2} A_{m-1} ^T A_m A_1 ^T A_2 A_3 ^TA_4 \cdots A_{m-2}^T A_{m-1} A_m ^T  )
	$$
	\label{th5H}
\end{theorem}

\subsection{Hadamard product of characteristic maxpolynomials}
\begin{theorem}
	\label{thm:Hdprod}
	Given matrices $A, B$ of order $n$ over $\Rmax$, let $p(x)=\fullchi{A}(x)$, $q(x)=\fullchi{B}(x)$ be the corresponding full characteristic maxpolynomials and
	let $(p \circ q)(x) =  \sum_{k=0}^n d_kx^k$ be the Hadamard product of the maxpolynomials.
	\begin{enumerate}[(i)]
		\item The Hadamard product can be written as
		\begin{equation}
		\label{eq:Hdproduct}
		(p \circ q)(x) = \bigoplus_{P,Q \in \cP_n}  \fullchi{A \circ PBQ}(x).
		\end{equation}
		\item\label{it2:thm:Hdprod}
		The ordered vector of the roots of $(p \circ q)(x)$ is the
                Hadamard product of the ordered vectors of the roots of $p(x)$ and $q(x)$.
	\end{enumerate}
\end{theorem}
\begin{proof}
	The proof is similar to the one of Theorem~\ref{thm:maxconv}, in fact even simpler. 
	When computing the coefficients of $(p \circ q)(x)$ then instead of computing permanent minors of a maximum of submatrices they are computed on (standard) sums of submatrices.
	Thus, each coefficient $c_{n-k}$ of $\bigoplus_{P,Q \in \cP_n} \fullchi{A \circ PBQ}(x) = \sum_{k=0}^n c_kx^k$ is obtained as the (standard) sum of a maximal permanent minor of order $k$ of $A$ and a maximal permanent minor of order $k$ of $B$, where the permutation matrices $P$ and $Q$ make sure that the positions of the elements of $B$ that contribute to the maximal permanent are mapped to the exact positions of the elements of $A$ that contribute to the maximal permanent minor.
	
	To be precise, let $p(x) = \sum_{k=0}^n a_kx^k$ and $q(x)=\sum_{k=0}^n b_k x^k$.
	By \eqref{eq:fullcharcoeff2} each coefficient $c_{n-k}$ is
	\begin{align*}
	c_{n-k}
	&= \bigoplus_{P,Q\in \cP_n} \eta_k(A\circ PBQ)\\
	&= \bigoplus_{\sigma,\tau\in\SG_n}
	\bigoplus_{\substack{i_1,i_2,\dots,i_k\\ j_1,j_2,\dots,j_k}} 
	a_{i_1j_1}b_{\sigma(i_1)\tau(j_1)}
	a_{i_2j_2}b_{\sigma(i_2)\tau(j_2)}
	\dotsm
	a_{i_kj_k}b_{\sigma(i_k)\tau(j_k)}
	\\
	&= \bigoplus_{\substack{i_1,i_2,\dots,i_k\\ i_1',i_2',\dots,i_k'\\ j_1,j_2,\dots,j_k\\ j_1',j_2',\dots,j_k'}} 
	a_{i_1j_1} a_{i_2j_2} \dotsm a_{i_kj_k}
	b_{i_1'j_1'} b_{i_2'j_2'} \dotsm b_{i_k'j_k'}
	\\
	&= \eta_k(A) \,\eta_k(B)
	\\
	&= a_{n-k} \, b_{n-k},
	\end{align*}
	which is $d_{n-k}$, the $(n-k)$th coefficient of $(p \circ q)(x)$.

	Property \ref{it2:thm:Hdprod} follows from  Proposition~\ref{prop:hadamardroots}.
\end{proof}

\subsection{Hadamard product of characteristic maxpolynomials via multiplicative convolution}
When trying to form the analogue of Theorem~\ref{thm:maxconv} with matrix multiplication instead of summation (max), that is, using an expression of the form $\bigoplus_{P,Q \in \cP_n}  \fullchi{A PBQ}(x)$, we realize that it cannot be done in general and that we have to restrict ourselves to specific classes of matrices.
The problem lies in the fact that when performing matrix multiplication we perform a series of scalar products of row vectors by column vectors, and these operations depend on the order of the elements in each vector.
Specifically, the scalar product is maximal only when the maximal element in each of the vectors is in the same position.

Hence, it is desired that the matrices $A$ and $B$ match with regard to the positions of the maximal elements in the rows of $A$ and the columns of $B$.
For example, $A$ and $B$ match when the maximal elements of the rows of $A$ lie in different columns and the maximal elements of the columns of $B$ lie in different rows (in case there are more than one maximal element in a row of $A$ (resp. a column of $B$), each of these elements is a legitimate choice).
More generally, it is necessary that the matrices $A^T$ and $B$ have the same \textbf{max-column partition}, which is defined as follows.
Given a matrix $M \in \Rnnmax$, for each $j$, $j = 1,\ldots,n$, let $m_j$ be the $j$-th column of $M$ and let $m_{i_j j} = \norm{m_j}$ be the maximal element of this column (if the maximum is attained in more than one place then we have more than one max-column partition associated with $M$).
Let $\sigma \in \SG_n$ be a permutation which arranges the maximal elements of the columns in ascending order:
$$
m_{i_{\sigma(1)} \sigma(1)} \leq m_{i_{\sigma(2)} \sigma(2)} \leq \cdots \leq m_{i_{\sigma(n)} \sigma(n)}.
$$
Then a max-column partition of $M$ is a partition of $[n]$ into blocks such that $j$ and $k$ lie in the same block if the corresponding matrix elements according to the above order lie in the same row:
$$
j \sim k \quad \Longleftrightarrow \quad i_{\sigma(j)} = i_{\sigma(k)}.
$$
For example, at the bottom of the lattice of partitions is the one where the maximal elements of the columns of $M$ lie in distinct rows: $i_j \neq i_k$ for each $i \neq k$. The max-column partition of $M$ consists then of $n$ blocks, where each block is a singleton.
On the other hand, the top partition is the one where there is a single block with $n$ elements, corresponding to the case where the maximal elements of the columns belong to the same row: $i_1 = i_2 = \cdots = i_n$.

Given two matrices $A, B \in \Rnnmax$, such that $A^T$ and $B$ share a max-column partition, we show now that the Hadamard product of the Gram characteristic maxpolynomials of $A^T$ and $B$ equals the maximum ($\oplus$), over all permutation matrices $P \in \cP_n$, of the set of full characteristic maxpolynomials of $APB$.
This maximum is achieved on a specific permutation matrix $P_0$ which ``orients'' $B$ towards $A$ by rearranging the rows of $B$.
In addition, when we allow multiplication on the right of $B$ with permutation matrices $Q \in \cP_n$ then we can restrict ourselves to the set of characteristic maxpolynomials instead of full characteristic maxpolynomials.
Here the orientation of $B$ towards $A$ is achieved through two specific permutation matrices $P_0$ and $Q_0$, which rearrange the rows as well as the columns of $B$.
\begin{theorem}
  Let $A, B \in \Rnnmax$ be two matrices, such that $A^T$ and $B$ share a max-column partition.
  Let $p(x)=\chi_{\widehat{A^T}}(x), q(x)=\chi_{\widehat{B}}(x)$ be the Gram characteristic maxpolynomials of $A^T, B$, respectively.
  Then
  \be
  (p \circ q)(x) = \bigoplus_{P \in \cP_n} \fullchi{A P B}(x)
  = \bigoplus_{P,Q \in \cP_n} \chi_{A P B Q}(x).
  \label{thm:maxcoloartition}
  \ee
  Moreover, there exist permutation matrices $P_0, Q_0 \in \cP_n$ such that
  $$
  (p \circ q)(x) = \fullchi{A P_0 B}(x) = \chi_{A P_0 B Q_0}(x).
  $$
  In addition, the vector of ordered roots of $(p \circ q)(x)$ equals the Hadamard product of the vector of ordered roots of $p(x)$ and the vector of ordered roots of $q(x)$.
  \label{thm:plusconv}
\end{theorem}
\begin{proof}
  We start with the polynomial
  $$
  \bigoplus_{k=0}^n c_k x^k :=   \bigoplus_{P \in \cP_n} \fullchi{A P B}(x)
  $$
  By   \eqref{eq:fullcharcoeff2} the coefficients are given by
  \begin{align*}
    c_{n-k}
    &= \bigoplus_{P \in \cP_n} \eta_k(A P B)\\
    &= \bigoplus_{\pi\in\SG_n} \bigoplus_{\substack{i_1,i_2,\dots,i_k\\
    j_1,j_2,\dots,j_k}} \bigoplus_{l_1,l_2,\dots,l_k} a_{i_1l_1}b_{\pi(l_1)j_1}
    a_{i_2l_2}b_{\pi(l_2)j_2}\dotsm a_{i_kl_k}b_{\pi(l_k)j_k}
    ;
    \intertext{each summation can be estimated with the Cauchy-Bunyakovsky-Schwarz 
      inequality  
      \eqref{eq:cauchy}
    }
    &\leq  \bigoplus_{\substack{i_1,i_2,\dots,i_k\\
    j_1,j_2,\dots,j_k}} \norm{a_{i_1}} \norm{b_{j_1}} \norm{a_{i_2}}
    \norm{b_{j_2}} \dotsm \norm{a_{i_k}} \norm{b_{j_k}},
    \intertext{where by $a_i$ we denote the rows of $A$ and by $b_j$ the columns of $B$.
    Now by   \eqref{eq:grammaxminor} this maximum is
    }
    &= \eta_k(\widehat{A^T})\,\eta_k(\widehat{B})
  \end{align*}
  and we have a chain of equalities interrupted by one inequality.
  In order for equality to hold in this inequality, equality must hold
  in each Cauchy-Bunyakovsky-Schwarz inequality, which boils down to the
  requirement that $A^T$ and $B$ share a max-column partition.

Moreover, there exists a permutation matrix $P_0$ such that 
	$$
            (p\circ q)(x)=		\bigoplus_{P \in \cP_n} \fullchi{A P B}(x) = \fullchi{A P_0B}(x).
	$$
	
	Indeed, the permutation matrix $P_0$ should arrange the rows of $B$ to match the positions of the maximal elements in the rows of $A$. Let $r_1 \geq r_2 \geq \cdots \geq r_n$ be the maximal elements of the rows of $A$ (the columns of $A^T$), that is, the roots of $\chi_{\widehat{A^T}} (x)$, the Gram  characteristic maxpolynomial of $A^T$, and let $s_1 \geq s_2 \geq \cdots \geq s_n$ be the maximal elements of the columns of $B$. Since  $A^T$ and $B$ share a max-column partition, there exists a permutation matrix $P_0$ such that if 
$r_i$ is in column $k_i$ then $s_i$ is in row $k_i$ of $P_0B$, for $i = 1, \ldots, n$.

	The elements $t_j = r_j+s_j$ lie in $n$ different rows and $n$ different columns of $AP_0B$.
	By multiplying on the right with the appropriate permutation matrix $Q_0$, these elements can be moved to the diagonal and it follows that $\fullchi{A P_0 B}(x) = \chi_{A P_0 B Q_0}(x)$.
	
	By Theorem~\ref{thm:fullsemi} and Proposition~\ref{prop:hadamardroots} the set of roots of $(p \circ q)(x)$ is the Hadamard product of the roots of $p(x)$ and the roots of $q(x)$.
	In fact, it is easily verified that $p(x)$ and $q(x)$ are FCF maxpolynomials without the need for Theorem~\ref{thm:fullsemi} since one can treat $\widehat{A^T}$ and $\widehat{B}$ as diagonal matrices, as the elements $t_j$ are the only elements that contribute to the characteristic maxpolynomials $p(x)$ and $q(x)$.
	\end{proof}
\begin{remark}
	It is clear that the condition in Theorem~\ref{thm:plusconv} about $A^T$ and $B$ sharing a max-column partition is not only sufficient but also necessary for equality \eqref{thm:maxcoloartition} to hold.
\end{remark}
\begin{example}
	Let $A, B$ be the matrices
	$$
	A =
	\begin{bmatrix*}[r]
	2 &  0 & 3 & -1 \\
	0 & 0 & 1 & 1 \\
	-2 & 2 & 2 & 1 \\
	2 & -1 & 1 & 1	
	\end{bmatrix*},
	\qquad
	B =
	\begin{bmatrix*}[r]
 	 0 & 0 & -2 & 2 \\
	-2 & 1 & -1 & -1 \\
	-1 & 0 & -3 & -1 \\
	-1 & -2 & -1 & 0	
	\end{bmatrix*}.
	$$
	The Gram characteristic maxpolynomial of $A^T$ is the characteristic maxpolynomial of
	$$
	\widehat{A^T} = (AA^T)^{\circ \half} = 
		\begin{bmatrix*}[r]
                  3 &  2 & \frac{5}{2} & 2 \\
                  \midrule[0pt]
		2 & 1 & \frac{3}{2} & 1 \\
                  \midrule[0pt]
		\frac{5}{2} & \frac{3}{2} & 2 & \frac{3}{2} \\
                  \midrule[0pt]
		2 & 1 & \frac{3}{2} & 2	
	\end{bmatrix*},
	$$
	which is
	\begin{align*}
	p(x) &= x^4 \oplus 3x^3 \oplus 5x^2 \oplus 7x \oplus 8  \\
	&= (x \oplus 3) (x \oplus 2)^2 (x \oplus 1).
	\end{align*}
	Then
	$$
	\widehat{B} = (B^TB)^{\circ \half} = 
	\begin{bmatrix*}[r]
		0 & 0 & -1 & 1 \\
		0 & 1 & 0 & 1 \\
		-1 & 0 & -1 & 0 \\
		1 & 1 & 0 & 2	
	\end{bmatrix*},
	$$
	whose  characteristic maxpolynomial is
	\begin{align*}
	q(x) &= x^4 \oplus 2x^3 \oplus 3x^2 \oplus 3x \oplus 2  \\
	&= (x \oplus 2)(x \oplus 1)(x \oplus 0)(x \oplus -1).
	\end{align*}
	The Hadamard product of $p(x)$ and $q(x)$ is
	\begin{align*}
	(p \circ q)(x) &= (0 \odot 0) x^4 \oplus (3 \odot 2) x^3 \oplus (5 \odot 3) x^2 \oplus (7 \odot 3) x \oplus (8 \odot 2) \\
	&=  x^4 \oplus 5x^3 \oplus 8x^2 \oplus 10x \oplus 10  \\
	&= (x \oplus 5)(x \oplus 3)(x \oplus 2)(x \oplus 0).
	\end{align*}
	We see that the roots of $(p \circ q)(x)$ are
	$$
	(5,3,2,0) = (3,2,2,1) \circ (2,1,0,-1),
	$$
	the Hadamard product of the ordered roots of $p(x)$ and $q(x)$.
	
	Let us now look at the maximal elements of the rows of $A$ and columns
        of $B$ (marked with an asterisk):
	$$
	A =
	\begin{bmatrix*}[r]
	2 &  0 & \markentry{3} & -1 \\
	0 & 0 & \markentry{1} & \markentry{1} \\
	-2 & \markentry{2} & \markentry{2} & 1 \\
	\markentry{2} & -1 & 1 & 1	
	\end{bmatrix*},
	\qquad
	B =
	\begin{bmatrix*}[r]
	\markentry{0} & 0 & -2 & \markentry{2} \\
	-2 & \markentry{1} & \markentry{-1} & -1 \\
	-1 & 0 & -3 & -1 \\
	-1 & -2 & \markentry{-1} & 0	
	\end{bmatrix*}.
	$$
	The ordered list of column-maximal elements in $B$ is $(-1,0,1,2)$, referring to columns $(3,1,2,4)$.
	The corresponding list of rows of these elements is $((2,4),1,2,1)$, where the pair $(2,4)$ refers to the maximal element of the third column, namely \\
$-1$, which occurs in row $2$ and in row $4$.
	We see that the matrix $B$ admits two max-column partitions.
	If we choose the second row in the third column then the partition is $\{(1,3),(2,4)\}$: the first and third ordered elements ($-1$ and $1$) lie in the same row (second row), whereas the second and fourth elements ($0$ and $2$) lie also in the same row (first row).
	The second partition is $\{(1),(3),(2,4))\}$, which is obtained by choosing the $4$-th row as the position of the maximal element of the third column.
	
	The matrix $A$ admits several max-row partitions (max-column partitions of $A^T$), including the partition $\{(1),(3),(2,4))\}$, which is also a max-column partition of $B$.
	The chosen maximal elements in the rows of $A$ (in ascending order) are $a_{24},a_{33},a_{41},a_{13}$, and the chosen maximal elements in the columns of $B$ are $b_{43},b_{11},b_{22},b_{14}$:
	$$
	A =
	\begin{bmatrix*}[r]
	2 &  0 & \markentry{3} & -1 \\
	0 & 0 & 1 & \markentry{1} \\
	-2 & 2 & \markentry{2} & 1 \\
	\markentry{2} & -1 & 1 & 1	
	\end{bmatrix*},
	\qquad
	B =
	\begin{bmatrix*}[r]
	\markentry{0} & 0 & -2 &  \markentry{2} \\
	-2 & \markentry{1} & -1 & -1 \\
	-1 & 0 & -3 & -1 \\
	-1 & -2 & \markentry{-1} & 0	
	\end{bmatrix*}.
	$$
	
	The list of rows of $A$ ordered by their maximal elements (in ascending order) is $(2,3,4,1)$, with corresponding columns $(4,3,1,3)$.
	The list of columns of $B$ ordered by their maximal element is $(3,1,2,4)$ with corresponding rows $(4,1,2,1)$.
	In order to match the positions of the chosen maximal elements of the rows of $A$ and the columns of $B$ we need to transfer  $(4,1,2,1)$ to $(4,3,1,3)$, that is, perform the moves $1 \to 3$ and $2 \to 1$.
	Hence we need to move the first row of $B$ to the third and to move the second row of $B$ to the first.
	This can be achieved via the permutation matrix $P_0$ that corresponds to the permutation $(2 \, 1 \, 3)$:
	$$
	P_0 B =
	\begin{bmatrix*}[r]
	\maxzero & 0 & \maxzero & \maxzero \\
	\maxzero & \maxzero & 0 & \maxzero \\
	0 & \maxzero & \maxzero & \maxzero \\
	\maxzero & \maxzero & \maxzero & 0	
	\end{bmatrix*}
	\begin{bmatrix*}[r]
	\markentry{0} & 0 & -2 & \markentry{2} \\
	-2 & \markentry{1} & -1 & -1 \\
	-1 & 0 & -3 & -1 \\
	-1 & -2 &  \markentry{-1} & 0	
	\end{bmatrix*}
	=
	\begin{bmatrix*}[r]
	-2 & \markentry{1} & -1 & -1 \\
	-1 & 0 & -3 & -1 \\
	\markentry{0} & 0 & -2 & \markentry{2} \\
	-1 & -2 &  \markentry{-1} & 0	
	\end{bmatrix*}.
	$$

	Then, multiplying $A$ with $P_0 B$ gives
	$$
	AP_0 B =
	\begin{bmatrix*}[r]
	2 &  0 & \markentry{3} & -1 \\
	0 & 0 & 1 & \markentry{1} \\
	-2 & 2 & \markentry{2} & 1 \\
	\markentry{2} & -1 & 1 & 1	
	\end{bmatrix*}
	\begin{bmatrix*}[r]
	-2 & \markentry{1} & -1 & -1 \\
	-1 & 0 & -3 & -1 \\
	\markentry{0} & 0 & -2 & \markentry{2} \\
	-1 & -2 &  \markentry{-1} & 0	
	\end{bmatrix*}
	=
	\begin{bmatrix*}[r]
	3 & 3 & 1 & \markentry{5} \\
	1 & 1 & \markentry{0} & 3 \\
	\markentry{2} & 2 & 0 & 4 \\
	1 & \markentry{3} & 1 & 3	
	\end{bmatrix*}.
	$$	
	
	The marked elements in $AP_0 B$ are the roots of the full characteristic maxpolynomial: $\fullchi{A P_0 B}(x) = (x \oplus 5)(x \oplus 3)(x \oplus 2)(x \oplus 0) = (p \circ q)(x)$.
	
	Finally, if we want the roots of the full characteristic maxpolynomial to lie on the diagonal (and then the characteristic maxpolynomial equals the full characteristic maxpolynomial), then we need to permute the columns of $AP_0 B$ by multiplying on the right with the matrix $Q_0$, which represents the permutation $(1 \, 3 \, 2 \, 4)$:
	$$
	AP_0 B Q_0 =
	\begin{bmatrix*}[r]
	3 & 3 & 1 & \markentry{5} \\
	1 & 1 & \markentry{0} & 3 \\
	\markentry{2} & 2 & 0 & 4 \\
	1 & \markentry{3} & 1 & 3	
	\end{bmatrix*}
	\begin{bmatrix*}[r]
	\maxzero & \maxzero & 0 & \maxzero \\
	\maxzero & \maxzero & \maxzero & 0 \\
	\maxzero & 0 & \maxzero & \maxzero \\
	0 & \maxzero & \maxzero & \maxzero	
	\end{bmatrix*}
	=
	\begin{bmatrix*}[r]
	\markentry{5} &  1 & 3 & 3 \\
	3 & \markentry{0} & 1 & 1 \\
	4 & 0 & \markentry{2} & 2 \\
	3 & 1 & 1 & \markentry{3}	
	\end{bmatrix*}.
	$$	
	Clearly, $\chi_{A P_0 B Q_0}(x) = \fullchi{A P_0 B}(x)$.
\end{example}

\subsection*{Acknowledgements.} 

\begin{small}
We thank Bettina Klinz for interesting and fruitful discussions.
Research of the first author was supported by the Austrian Science Fund (FWF) Projects P25510-N26 and P29355-N35.
The third author was supported in part by a JESH grant of Austrian Academy of Sciences  and he also acknowledges a  partial support of  the Slovenian Research Agency (grants P1-0222 and J1-8133). He thanks his colleagues and staff at TU Graz for their hospitality during his visits in Austria.
\end{small}
\bibliography{polynomial_max_convolutions}
\bibliographystyle{amsplain}
\end{document}